\documentclass{amsart}
\usepackage{amsmath,amssymb,verbatim}
\usepackage{enumerate}
\newtheorem{theorem}{Theorem}[section]
\newtheorem{lemma}[theorem]{Lemma}
\newtheorem{proposition}[theorem]{Proposition}
\newtheorem{corollary}[theorem]{Corollary}
\newtheorem{claim}[theorem]{Claim}
\theoremstyle{definition}

\theoremstyle{remark}
\newtheorem{remark}{Remark}[section]

\DeclareMathOperator{\sech}{sech}
\DeclareMathOperator{\csch}{csch}
\DeclareMathOperator{\spann}{span}

\newcommand{\R}{\mathbb{R}}
\newcommand{\C}{\mathbb{C}}

\newcommand{\Z}{\mathbb{Z}}
\newcommand{\la}{\langle}
\newcommand{\ra}{\rangle}
\newcommand{\pd}{\partial}
\newcommand{\eps}{\epsilon}
\newcommand{\mF}{\mathcal{F}}

\newcommand{\bQ}{\mathbf{Q}}
\newcommand{\bQd}{\mathbf{Q'}}
\newcommand{\bR}{\mathbf{R}}
\newcommand{\bRd}{\mathbf{R'}}

\newcommand{\tg}{\tilde{g}}
\renewcommand{\a}{\alpha} 
\renewcommand{\k}{\kappa}
\begin{document}
\title{Transverse linear stability of line solitons for 2D Toda}
\author{Tetsu Mizumachi}
\address{Division of Mathematical and Information Sciences,
Hiroshima University, Kagamiyama 1-7-1, 739-8521 Japan}
\email{tetsum@hiroshima-u.ac.jp}
\keywords{$2$D Toda, line solitary waves, transverse linear stability}
\subjclass[2020]{Primary, 35B35, 37K40; Secondary, 35Q51}
\begin{abstract}
The $2$-dimensional Toda lattice ($2$D Toda) is a completely integrable
semi-discrete wave equation with the KP-II equation in its continuous limit.
Using Darboux transformations, we prove the linear stability of
$1$-line solitons for $2$D Toda of any size in an exponentially weighted space.
We prove that the dominant part of solutions to the linearized equation around
a $1$-line soliton is a time derivative of the $1$-line soliton multiplied by
a function of time and transverse variables. The amplitude
is described by a $1$-dimensional damped wave equation in the transverse
variable, as is the case with the linearized KP-II equation.
\end{abstract}
\maketitle
\section{Introduction}
\label{sec:intro}
In this paper, we consider transverse linear stability of
$1$-line solitons for the $2$-dimensional Toda lattice equation
\begin{equation}
  \label{eq:2dToda}
  \pd_x\pd_s\log(1+V_n)=V_{n+1}-2V_n+V_{n-1}
\end{equation}
where $V_n=V_n(s,x)$ and $(n,s,x)\in\Z\times \R^2$\,.

\par
By the change of variables
\begin{equation}
  \label{eq:v-change}
t=x+s\,,\qquad y=x-s\,,\qquad R_n=\log(1+V_n)\,,  
\end{equation}
Eq.~\eqref{eq:2dToda} is translated into 
\begin{equation}
  \label{eq:2dToda-h}
(\pd_t^2-\pd_y^2)R_n=e^{R_{n+1}}-2e^{R_n}+e^{R_{n-1}}\,,
\end{equation}
where $R_n=R_n(t,y)$ and $(n,t,y)\in\Z\times\R^2$.
If $R_n$ is independent of $y$, then $R_n$ is a solution of the Toda
lattices
\begin{equation}
  \label{eq:1d-Toda}
  \frac{d^2R_n}{dt^2}=e^{R_{n+1}}-2e^{R_n}+e^{R_{n-1}}\,.
\end{equation}
Both \eqref{eq:2dToda} and \eqref{eq:1d-Toda} are integrable.
See, e.g. \cite{Flaschka, Hirota, MOP}.
\par
Using $\tau$-functions, we can rewrite \eqref{eq:2dToda} in Hirota's bilinear
form:
\begin{gather}
\label{eq:tau}
1+V_n=\tau_{n+1}\tau_{n-1}/\tau_n^2\,,
  \\
\label{eq:bilinear}
D_sD_x\tau_n\cdot\tau_n=2(\tau_{n+1}\tau_{n-1}-\tau_n^2)\,,
\end{gather}
where
\begin{equation*}
  D_sD_x\tau_n\cdot\tau_n:=\pd_{\eps_1}\pd_{\eps_2}
\tau_n(s+\eps_1,x+\eps_2)\tau_n(s-\eps_1,x-\eps_2)\bigr|_{\eps_1=\eps_2=0}\,.  
\end{equation*}
If $\pd_x\phi_i=e^\pd \phi_i$ and $\pd_s\phi_i=-e^{-\pd}\phi_i$
for $i=1$, $\ldots$, $N$,
\begin{gather*}
\tau_n=
\begin{vmatrix}
  \phi_1(n) & \phi_1(n+1) & \ldots & \phi_1(n+N-1)\\
  \vdots & \vdots & \vdots  & \vdots \\
  \phi_N(n) & \phi_N(n+1) & \ldots & \phi_N(n+N-1)
\end{vmatrix}
\end{gather*}
satisfies \eqref{eq:bilinear}. 
See e.g. \cite{Hirota}. 
Especially, if $a\in\R\setminus\{0\}$, $\k=\log|a|$ and 
\begin{equation}
  \label{eq:tau-1}
\tau_n=\phi_1(n)=a^ne^{ax-s/a}+a^{-n}e^{x/a-as}
\quad\text{for $(n,s,x)\in\Z\times\R^2$,}
\end{equation}
we have a $1$-line soliton solution
\begin{equation}
  \label{eq:1sol}
V_n=
\begin{cases}
  &  \sinh^2\k\sech^2\left(n\k-t\sinh\k\right)=:V^\k_n
  \quad\text{if $a<0$,}
  \\ &
  \sinh^2\k\sech^2\left(n\k+t\sinh\k\right)
  \quad\text{if $a>0$,}
\end{cases}
\end{equation}
which tends to $0$ as $n\to\pm\infty$ and do not decay in the transversal
direction. Note that $R^\k_n:=\log(1+V^\k_n)$ is a $1$-soliton solution of
\eqref{eq:1d-Toda}.
For multi-line soliton solutions of \eqref{eq:2dToda} and
their classification, see e.g. \cite{Biondini-Kodama,Biondini-Wang}.
\par
The $1$-dimensional Toda lattice equation \eqref{eq:1d-Toda}
has the KdV equation in its continuous limit (\cite{Bambusi-Kappeler-Paul}).
The stability of the soliton solutions of \eqref{eq:1d-Toda} was studied using
the PDE method (\cite{BHW,FP2,FP3,Miz09,MP}) and the nonlinear steepest
descent method (\cite{Kruger-Teschl-1,Kruger-Teschl-2}).
\par
In this paper, we study the transverse linear stability of \eqref{eq:1sol}
as a solution of the $2$-dimensional Toda lattices \eqref{eq:2dToda-h}.
If we linearize \eqref{eq:2dToda-h} around $R=\log(1+V^\k)$, we have
\begin{equation}
  \label{eq:linear-1}
(\pd_t^2-\pd_y^2)\bRd=(e^\pd-2+e^{-\pd})(1+V^\k)\bRd\,,
\end{equation}
where $e^{\pm\pd}$ are shift operators defined by
$e^{\pm\pd}f(n,s,x)=f(n\pm1,s,x)$.
\par
We will prove that a solution of \eqref{eq:linear-1} satisfying a
secular term condition decays exponentially in a weighted space
whose weight function increases exponentially as $n\to\infty$.

Let $\a\in\R$ and $\ell^2_\a L^2$ be a complex Hilbert
space with an inner product
\begin{equation*}
(f,g)_{\ell^2_\a L^2}  
=\sum_{n\in\Z}e^{2\a n}\int_\R f(n,y)\overline{g(n,y)}\,dy\,,
\end{equation*}
and let $\ell^2_\a H^1$ be a Hilbert space with a norm
$$\|f\|_{\ell^2_\a H^1}
=\left(\|\pd_yf\|_{\ell^2_\a L^2}^2+\|f\|_{\ell^2_\a L^2}^2\right)^{1/2}\,.$$
\par
Let $\tg^{\pm,*}(t,y,\eta)$ be solutions to
the adjoint equation of \eqref{eq:linear-1} defined by
\eqref{eq:tg1*form} and \eqref{eq:tg2*form}, respectively and
\begin{equation}
  \label{def:eta*}
\eta_*(\a)=\tanh(\k+\a)\sqrt{\sinh\a\sinh(2\k+\a)}\,.
\end{equation}
\begin{theorem}
  \label{thm:main}
  Let $\a\in(0,2\k)$ and $c=\sinh\k/\k$.
Suppose that $\eta_0\in(0,\eta_*(\a))$ and $t_0\in\R$. If
  $\bRd(t)$ is a solution of \eqref{eq:linear-1}
  in the class $C(\R;\ell^2_\a H^1(\R))\cap C^1(\R;\ell^2_\a L^2(\R))$
  and if for every $\eta\in[-\eta_0,\eta_0]$, 
  \begin{equation*}
    \sum_{n\in\Z}\int_\R
    \left(\bRd(t)\overline{\pd_t\tg^{\pm,*}_n(t,y,\eta)}
      -\pd_t\bRd(t)\overline{\tg^{\pm,*}_n(t,y,\eta)}\right)\,
    dy=0
  \end{equation*}
holds at $t=t_0$, then for every $t\ge t_0$, 
  \begin{align*}
    & e^{-\a ct}\left(
\|\bRd(t)\|_{\ell^2_\a H^1(\R)}+ \|\pd_t\bRd(t)\|_{\ell^2_\a L^2(\R)}\right)
\\  \le & K e^{-b (t-t_0)}e^{-\a ct_0}
    (\|\bRd(t_0)\|_{\ell^2_\a H^1(\R)}
    + \|\pd_t\bRd(t_0)\|_{\ell^2_\a L^2(\R)})
  \end{align*}
where $K$ and $b$ are positive constants that are independent of
$t$, $t_0$ and $\bRd$.
\end{theorem}
\begin{remark}
 If $\eta$ is close to $0$,
\begin{align*}
& \spann\{\tg^{+,*}(\eta),\tg^{-,*}(\eta)\} \simeq 
e^{iy\eta}(e^\pd-2+e^{-\pd})^{-1} \spann\{\pd_tR^\k\pd_\k R^\k\}\,.
\end{align*}
\end{remark}
\begin{remark}
The exponentially weighted space is useful to observe that
a $1$-line soliton $V^\k$ moves faster to the direction $(n,y)=(1,0)$
than any solutions of \eqref{eq:linear-1}.
The idea was first applied to prove the asymptotic stability of
solitary waves for generalized KdV equaitons (\cite{PW}).
Their idea works not only for solitary waves of $1$-dimensional long
wave models but also for line solitary waves for long wave models
(\cite{Miz15,Miz18,Miz19,Miz-Shimabukuro20}).
\par  
  A $1$-line soliton $V_n=\sinh^2\k\sech^2(n\k+t\sinh\k)$ is
linearly stable in $\ell^2_{-\a}H^1(\R)\times \ell^2_{-\a}L^2(\R)$.
\end{remark}
\begin{remark}
Formally, \eqref{eq:2dToda-h} has the KP-II equation 
in its continuous limit, whereas the $2$-dimensional Toda equation
of elliptic type has the KP-I equation in its continuous limit.
The lump solutions for elliptic $2$d-Toda have been studied by \cite{YLiu-JWei-Toda1,YLiu-JWei-KP1,YLiu-JWei-Wang-KP1}  using the B\"acklund transformations.
\end{remark}
\begin{remark}
If \eqref{eq:2dToda-h} is discreitzed in $y$, then the equation is not
integrable anymore. For such equations,
metastability of small line solitary waves has been proved by
\cite{Hristrov-Pel,Pel-Sch}.
\end{remark}
To prove Theorem~\ref{thm:main}, we will use Darboux transformations to compare
the rate of decay of solutions to \eqref{eq:linear-1} and that
of solutions to the linearized  equation around $0$:
\begin{equation}
  \label{eq:linear-0}
  (\pd_t^2-\pd_y^2)\bR=(e^\pd-2+e^{-\pd})\bR
  \quad\text{on $\Z\times\R^2$.}
\end{equation}
\par
The B\"acklund transformations have been used to prove
nonlinear stability of solitons or breathers
for $1$-dimensional equations such as KdV, NLS and the
sine-Gordon equation in $L^2(\R)$ or in less regular spaces
since a B\"acklund transformation gives an isomorphism between
a neighborhood of $0$ and a neighborhood of solitons or breathers
in various norms.
See e.g. \cite{Alejo-Munoz13,Alejo-Munoz15,Contreras-Pelinovsky,
CPS,Koch-Tataru,MPel,MT,MV}.
\par

Darboux transformations map solutions of the linearized equation
around $N$-soliton that satisfy a secular term condition
to those of the linearize equation around $(N-1)$-soliton.
This property has been used to prove linear stability of
soliton solutions for solitons of $1$-dimensional Toda lattices, KdV
and KP-II (\cite{Miz13,Miz15,Miz23,MP,BHW}).  For $1$-dimensional
equations such as KdV and \eqref{eq:1d-Toda},
the secular modes arise from the symmetry of
soliton solutions, and they are finite dimensional.
\par
On the other hand, if we investigate the spectrum of the linearized operator
around a line solitary for long wave models of $3$D gravity water waves
in an exponentially weighted space, we find that a curve of
continous spectrum lies on the stable half plane and goes through $0$
(\cite{Gey-Pel-LiuY,Miz15,Miz-Shimabukuro17,Miz23,Rousset-CSun}).
\par
Being different from linearized equations for continuous models,
the linearized equation \eqref{eq:linear-1} is non-autonomous.
Nevertheless, we can find solutions of \eqref{eq:linear-1} that decay like
$O(e^{-t\delta})$ for any $\delta>0$ as is the case with the linearized
KP-II equation.
We will impose a secular term condition to exclude these slowly decaying
solution of \eqref{eq:linear-1} and show that via
the Darboux transformation,
solutions of \eqref{eq:linear-1} satisfying the secular term
condition are connected to solutions of \eqref{eq:linear-0}.
As a result, we find that solutions that do not include any secular modes
tend to $0$ at the same rate as solutions of \eqref{eq:linear-0}.
\par
Slowly decaying secular modes of linearized equations are hazardous to
use B\"{a}cklund to prove nonlinear stability of $1$-line solitons.
Recently, $1$-codimensional stability of $1$-line solitons of the
KP-II equation has been proved by \cite{Pompili} for perturbations in
critical function spaces (\cite{HHK}), which exclude low frequency
waves and prevent phase shifts of line solitons.  However, the idea of
\cite{Pompili} cannot be applied directly to \eqref{eq:2dToda-h}
because the equation is discrete in $n$.  We expect that our linear
stability result (Theorem~\ref{thm:main}) is a key step to prove
nonlinear stability of $1$-line solitons for \eqref{eq:2dToda-h} by
PDE methods as it is for the KP-II equation and the Benney-Luke equation
(\cite{Miz15, Miz18, Miz19, Miz-Shimabukuro20}.
\par
The secular modes corresponding to the local speeds are the dominant
part of solutions for \eqref{eq:linear-1} as $t\to\infty$ and
a dissipative wave equation describes their time evolution.
\begin{theorem}
  \label{thm:profile}
Let $\a\in(0,2\k)$, $\lambda_1=\coth\k-1/\k$ and $\lambda_2=(\sinh2\k/2\k-1)/2(\sinh\k)^3$.
Suppose that $\bRd$ is a solution  of \eqref{eq:linear-1} with
$$(\bRd(0),\pd_t\bRd(0))\in
(\ell^2_\a H^1(\R)\times\ell^2_\a L^2(\R))\cap (\ell^2_\a L^1(\R)
\times \ell^2_\a L^1(\R))\,.$$ Then there exists $f\in L^1(\R)$ such that
as $t\to\infty$,
\begin{equation*}
  \left\|
    \begin{pmatrix}
    \bRd(t) \\ \pd_t\bRd(t)  
    \end{pmatrix}
     - (H_t*W_t*f)(y)
     \begin{pmatrix}
     \pd_tR^\k \\ \pd_t^2R^\k
     \end{pmatrix}
     \right\|_{\ell^2_\a H^1(\R)\times \ell^2_\a L^2(\R)}
  =O(t^{-1/4})\,,
\end{equation*}
where $H_t(y)=(4\pi\lambda_2t)^{-1/2}e^{-y^2/4\lambda_2t}$ and
$W_t(y)=(2\lambda_1)^{-1}$ for $y\in [-\lambda_1 t, \lambda_1 t]$
and $W_t(y)= 0$ otherwise.
\end{theorem}
\par
Our plan for the present paper is as follows.
In Section~\ref{sec:preliminaries},
we express secular modes of \eqref{eq:linear-1} as products of
Jost functions and dual Jost functions for a Lax pair
with a $1$-line soliton potential.
A Darboux transformation provides a correspondence between the solutions of
\eqref{eq:linear-1} and those of \eqref{eq:linear-0}.
In Section~\ref{sec:Darboux}, we will formulate the correspondence
by using Jost functions and dual Jost functions.
In Section~\ref{sec:exp-st0}, we will prove linear
stability of $0$.  In Section~\ref{sec:Darboux-2}, we will estimate
fundamental solutions of Darboux transformations.
In Section~\ref{sec:proof}, we will prove Theorems~\ref{thm:main}
and \ref{thm:profile}.
\bigskip

\section{Jost functions and secular modes}
\label{sec:preliminaries}
Using shift operators, we can rewrite \eqref{eq:2dToda} as
\begin{equation}
  \label{eq:2dToda'}
\pd_x\pd_s\log(1+V)=(e^\pd-2+e^{-\pd})V\quad\text{on $\Z\times\R^2$,}
\end{equation}
where $V=V(n,s,x)$. Let
\begin{equation}
  \label{eq:Laxpair}
  L_1=\pd_s+(1+V)e^{-\pd}\,,\quad
  L_2=\pd_x-e^\pd-e^\pd(\pd_xq)e^{-\pd}\,.
\end{equation}
The Lax pair $(L_1,L_2)$ satisfies the compatibility condition $[L_1,L_2]=0$
if $V$ is a solution of \eqref{eq:2dToda'}.
\par
We say $\Phi(n,s,x)=\Phi_n(s,x)$ is a Jost function for the Lax pair
$L_1$, $L_2$ if $L_1\Phi= L_2\Phi=0$. That is,
\begin{equation}
  \label{eq:Jost-sol}
  \pd_s\Phi_n+(1+V_n)\Phi_{n-1}=0\,,\quad
  \pd_x\Phi_n-\Phi_{n+1}-\pd_xq_{n+1}\Phi_n=0\,.
\end{equation}
We say $\Phi^*=\{\Phi_n^*\}_{n\in\Z}$ is a dual Jost function for
the Lax pair $L_1$, $L_2$  if
$(L_1)^*\Phi^*= (L_2)^*\Phi^*=0$. That is,
\begin{equation}
    \label{eq:dualJost-sol}
  \pd_s\Phi_n^*-(1+V_{n+1})\Phi_{n+1}^*=0\,,\quad
  \pd_x\Phi_n^*+\Phi_{n-1}^*+\pd_xq_{n+1}\Phi_n^*=0\,.
\end{equation}
Analogous to the KP-II equation, products of Jost and dual Jost functions are solutions of linearized Toda equations.
\begin{lemma}
  \label{lem:prod}
Assume \eqref{eq:Jost-sol}, \eqref{eq:dualJost-sol} and $[L_1,L_2]=0$.
Then
 \begin{equation*}
   \pd_x\pd_s(\Phi\Phi^*)
   =(e^\pd-1)(1+V)(1-e^{-\pd})\Phi\Phi^*\,.
 \end{equation*}
\end{lemma}
\begin{proof}
  By \eqref{eq:Jost-sol} and \eqref{eq:dualJost-sol},
  \begin{align*}
\pd_s(\Phi_n\Phi^*_n)=& (1+V_{n+1})\Phi_n\Phi_{n+1}^*
 -(1+V_n)\Phi_{n-1}\Phi_n^*\,,
  \end{align*}
  \begin{multline*}
\pd_x\pd_s(\Phi_n\Phi^*_n)=\sum_{j=0,1}(-1)^{j-1}
(1+V_{n+j})(\Phi_{n+j}\Phi_{n+j}^*-\Phi_{n+j-1}\Phi_{n+j-1}^*)
    \\  +\sum_{j=0,1}(-1)^{j-1}(1+V_{n+j})\Phi_{n+j-1}\Phi^*_{n+j}
\left\{\frac{\pd_xV_{n+j}}{1+V_{n+j}}
         +\pd_x(q_{n+j}-q_{n+j+1})\right\}\,.
  \end{multline*}
Since $[L_1,L_2]=0$, we have $\pd_xV_n/(1+V_n)=\pd_x(q_{n+1}-q_n)$ for every
$n\in\Z$, and
\begin{equation*}
\pd_x\pd_s(\Phi_n\Phi^*_n)=\sum_{j,k=0,1}(-1)^{j+k-1}
(1+V_{n+j})\Phi_{n+j-k}\Phi_{n+j-k}^*\,.
\end{equation*}
Thus we prove Lemma~\ref{lem:prod}.
\end{proof}
\par
Next, we will introduce Jost functions and dual Jost functions for
\eqref{eq:Laxpair} when $V=0$ and when $V$ is a $1$-line soliton.
With an abuse of notation, let
\begin{gather}
  \notag
L_1=\pd_s+e^{-\pd}=\pd_t-\pd_y+e^{-\pd}\,,\quad
L_2=\pd_x-e^\pd=\pd_t+\pd_y-e^\pd\,,
\\  \label{def:Phi0}
 \Phi^0(\beta)=\{\Phi^0_n(\beta)\}_{n\in\Z}\,,\quad
 \Phi^0_n(\beta)=\beta^ne^{\beta x-s/\beta} \,,
 \\ \label{def:Phi0*}
 \Phi^{0,*}(\beta)=\{\Phi^{0,*}_n(\beta)\}_{n\in\Z}\,,\quad
 \Phi^{0,*}_n(\beta)=\beta^{-n}e^{-\beta x+s/\beta}\,.
\end{gather}
Then for any $\beta\in\C\setminus\{0\}$,
\begin{equation}
  \label{eq:Jost-0}
  L_1\Phi^0(\beta)=L_2\Phi^0(\beta)=0\,,\quad
 L_1^*\Phi^{0,*}(\beta)=L_2^*\Phi^{0,*}(\beta)=0\,.
\end{equation}
\par
Let $(L_1',L_2')$ be a Lax pair such that $[L_1',L_2']=0$ and
\begin{gather*}
L_1'=\pd_s+(1+V^\k)e^{-\pd}\,,\quad
L_2'=\pd_x-e^\pd- e^\pd(\pd_xq')e^{-\pd}\,.
\end{gather*}
Let $a=-e^\k$, $\k>0$ and
\begin{equation}
  \label{def:taun'}
\begin{aligned}
\tau_n'=& a^na^{ax-s/a}+a^{-n}e^{x/a-as}
\\=& 2(-1)^ne^{-y\cosh\k}\cosh(n\k-t\sinh\k)\,.
\end{aligned}
\end{equation}
Then $\pd_x\pd_s\log|\tau'_n|=V^\k_n$ and $q_n'=\log(\tau_n'/\tau_{n-1}')$.
Jost functions and dual Jost functions for  $L_1'$ and $L_2'$ are
$\Phi(\beta)=\{\Phi_n(\beta)\}_{n\in\Z}$ with
\begin{equation}
\label{eq:Jost}
\Phi_n(\beta,s,x):=\beta^ne^{\beta x -s/\beta}
\left(\beta-\frac{\tau'_{n+1}}{\tau'_n}\right)\,,  
\end{equation}
and $\Phi^*(\beta)=\{\Phi_n^*(\beta)\}_{n\in\Z}$ with
\begin{align}
\label{eq:dualJost}
\Phi^*_n(\beta,s,x):=&\frac{\beta^{-n}}{(\beta-a)(\beta-a^{-1})}
e^{-\beta x +s/\beta}
\left(\beta-\frac{\tau'_n}{\tau'_{n+1}}\right)
\\=& \notag
\beta^{-n}e^{-\beta x +s/\beta}
\frac{(\beta-\pd_x)^{-1}\tau'_{n+1}}{\tau'_{n+1}}\,,
\end{align}
respectively. Especially,
\begin{gather}
 \label{eq:Jost-zero}
\Phi_n(a,s,x)=-\Phi_n(a^{-1},s,x)
=(-1)^{n+1}e^{-y\cosh\k}\sinh\k\sech(n\k-t\sinh\k)\,,
\\  \label{eq:dualJost-pole}
\frac{1}{a}
\operatorname{Res}_{\beta=a}\Phi^*_n(\beta,s,x)
=a\operatorname{Res}_{\beta=1/a}\Phi^*_n(\beta,s,x)
= \frac{1}{\tau_{n+1}'}\,.
\end{gather}
By a straightforward computation, we have the following.
\begin{lemma}
  \label{lem:Jost}
  \begin{gather*}
L_1'\Phi(\beta)=L_2'\Phi(\beta)=0\quad
\text{for every $\beta\in\C\setminus\{0\}$,}
\\
(L_1')^*\Phi^*(\beta)= (L_2')^*\Phi^*(\beta)=0\quad
\text{for every $\beta\in\C\setminus\{0,a,a^{-1}\}$,}
\\
L_1'\frac{1}{e^{\pd}\tau'}=L_2'\frac{1}{e^{\pd}\tau'}=0\,.
  \end{gather*}
\end{lemma}
See e.g. \cite{Ueno-Takasaki} for definitions of wave functions for
\eqref{eq:2dToda}.
\par
Now, we will introduce secular modes for \eqref{eq:linear-1}.
Let $\bQd=(e^\pd-1)^{-1}\mathbf{R'}$. Then 
  \begin{equation}
    \label{eq:6}
 (\pd_t^2-\pd_y^2)\bQd=(1-e^{-\pd})(1+V^\k)(e^\pd-1)\bQd\,.
  \end{equation}
\par
Let $\tau_n'$ be a $\tau$-function of $1$-soliton defined
by \eqref{def:taun'}, $\tau'=\{\tau'_n\}_{n\in\Z}$ and
\begin{gather}
\notag
\beta_\pm(\eta)=-w(\eta)\pm\mu(\eta)\,,\quad
w(\eta)=\cosh\k+i\eta\,,\quad
\mu(\eta)=\sqrt{w(\eta)^2-1}\,,
\\ \label{def:gpm}
g^+(\eta)=\frac{e^{-\pd}\Phi(\beta_+(-\eta))}{\tau'}\,,
\quad
g^-(\eta)=e^{-\pd} \left(\Phi(a)\Phi^*(\beta_-(\eta)) \right)\,,
\\ \label{def:g*pm}
g^{+,*}(\eta)=\frac{e^{-\pd}\Phi(\beta_-(-\eta))}{\tau'}\,,
\quad
g^{-,*}(\eta)=e^{-\pd}\left(\Phi(a)\Phi^*(\beta_+(\eta))\right)\,,
\\ \label{def:tgpm}
\tg^+(\eta)=\frac{\Phi^0(\beta_+(-\eta))}{\tau'}\,,
\quad \tg^-(\eta)=\frac{1}{2i\eta}
\Phi(a)\Phi^{0,*}(\beta_-(\eta))\,,
\\ \label{def:tgpm*}
\tg^{+,*}(\eta)=\frac{\Phi^0(\beta_-(-\eta))}{\tau'}\,,
\quad \tg^{-,*}(\eta)=\frac{1}{2i\eta}
\Phi(a)\Phi^{0,*}(\beta_+(\eta))\,.
\end{gather}
Let
\begin{gather*}
z_n(t)=n-t\sinh\k/\k\,,\quad
\gamma(\eta)=\log(-\beta_-(\eta))\,,\quad
\delta(\eta)=\frac{\sinh\k}{\k}\gamma(\eta)-\mu(\eta)\,.
\end{gather*}
Then we have the following.
\begin{lemma}
  \label{lem:tgform}
It holds that $g^\pm(\eta)$ and $g^{\pm,*}(\eta)$ are solutions of \eqref{eq:6}
and that
\begin{gather}
  \label{eq:g-tg}
  g^\pm(\eta)=(1-e^{-\pd})\tg^\pm(\eta)\,,
  \quad g^{\pm,*}(\eta)=(1-e^{-\pd})\tg^{\pm,*}(\eta)\,,
\\  
  \label{eq:tg1form}  
\tg^+_n(\eta)=\frac12e^{iy\eta-t\delta(-\eta)}
e^{-\gamma(-\eta)z_n(t)}\sech\k z_n(t)\,,
\\
\label{eq:tg1*form}
\tg^{+,*}_n(\eta)=
\frac12e^{iy\eta+t\delta(-\eta)}e^{\gamma(-\eta)z_{n}(t)}\sech\k z_{n}(t)\,,
\\
\label{eq:tg2form}
\tg^-_n(\eta)=-
\frac{\sinh\k}{2i\eta}e^{iy\eta-t\delta(\eta)}
e^{-\gamma(\eta)z_n(t)}\sech\k z_n(t)\,,  
\\
\label{eq:tg2*form}
\tg^{-,*}_n(\eta)=-\frac{\sinh\k}{2i\eta}e^{iy\eta+t\delta(\eta)}
e^{\gamma(\eta)z_n(t)}\sech\k z_n(t)\,.
\end{gather}
\end{lemma}
\begin{remark}
Taking the Fourier transform of \eqref{eq:6} with respect to $y$,
we have
\begin{equation}
 \label{eq:linear1-F}
 (\pd_t^2+\eta^2)\widehat{\bQd}(\eta)
 =(1-e^{-\pd})(1+V^\k)(e^\pd-1)\widehat{\bQd}(\eta)\,.
\end{equation}
Since $g^\pm(\eta)$ and  $g^{\pm,*}(\eta)$ are solutions of \eqref{eq:6}
and $e^{-iy\eta}g^\pm(\eta)$ and $e^{-iy\eta}g^{\pm,*}(\eta)$ are
independ of $y$, $g^\pm(\eta)$ and $g^{\pm,*}(\eta)$
satisfy \eqref{eq:linear1-F}.
\end{remark}

\begin{proof}[Proof of Lemma~\ref{lem:tgform}]
By Lemmas~\ref{lem:prod} and \ref{lem:Jost},
$g^\pm(\eta)$ and $g^{\pm,*}(\eta)$ are solutions of \eqref{eq:6}.
We have \eqref{eq:g-tg} from Claim~\ref{cl:Phi-updown}
since $\beta_\pm(\eta)^2+2(\cosh\k+i\eta)\beta_\pm(\eta)+1=0$.
\par
By \eqref{eq:v-change} and the definition of $\beta_\pm(\eta)$,
\begin{equation*}
  x\beta_\pm(\eta)-s\beta_\pm(\eta)^{-1}=-yw(\eta)\pm t\mu(\eta)\,.
\end{equation*}
Combining the above with \eqref{def:taun'} and \eqref{eq:Jost-zero},
we have \eqref{eq:tg1form}--\eqref{eq:tg2*form}.
We see from \eqref{eq:g-tg}--\eqref{eq:tg2*form} that
$e^{-iy\eta}g^\pm(\eta)$ and $e^{-iy\eta}g^{\pm,*}(\eta)$ are
independent of $y$.
\end{proof}
Since $\gamma(\eta)=\k+O(\eta)$,
$g^\pm(\eta)\in \ell^2_\a$ and $g^{\pm,*}(\eta)\in \ell^2_{-\a}$
provided $\eta$ is sufficiently small.
More precisely, we have the following.
\begin{lemma}
\label{lem:g-norm}  
Suppose that $\a\in(0,2\k)$ and that $\eta\in(-\eta_*(\a),\eta_*(\a))$.
Then
\begin{equation*}
  \|\tg^\pm(\eta)\|_{\ell^2_\a}=O(e^{-t\Re\delta(\eta)})\,,\quad
  \|\tg^{\pm,*}(\eta)\|_{\ell^2_\a}=O(e^{t\Re\delta(\eta)})\,.
\end{equation*}  
\end{lemma}
To prove Lemma~\ref{lem:g-norm}, we need the following.
\begin{claim}
  \label{cl:beta1-beta2}
  Let $\eta\in\R$,
  $\gamma_R(\eta)=\Re\gamma(\eta)$ and $\gamma_I(\eta)=\Im\gamma(\eta)$.
 Then $\gamma_R(\eta)$ is an even function and $\gamma_I$ is an odd function.
 Moreover, $\beta_+(\eta)$ and $\beta_-(\eta)$ satisfy the following.
\begin{itemize}
  \item[\rm{(i)}]
$\beta_+(\eta)\beta_-(\eta)=1$ and
$|\beta_+(\eta)|$ and $|\beta_-(\eta)|$ are even functions.
  \item[\rm{(ii)}]
$|\beta_+(\eta)|<-\beta_+(0)=e^{-\k}$ and
 $|\beta_-(\eta)|>-\beta_-(0)=e^{\k}$ for $\eta\in\R\setminus\{0\}$.
\item[\rm{(iii)}]
$|\beta_-(\eta)|$ is monotone increasing for $\eta>0$ and
$|\beta_-(\eta_*(\a))|=e^{\a+\k}$.
\end{itemize}
\end{claim}

Since $|\gamma_R(\pm\eta)-\a|<\k$ for $\eta\in(-\eta_*(\a),\eta_*(\a))$
by Claim~\ref{cl:beta1-beta2},
Lemma~\ref{lem:g-norm} follows from \eqref{eq:tg1form}--\eqref{eq:tg2*form}.

\begin{proof}[Proof of Claim~\ref{cl:beta1-beta2}]
By the definitions, we have (i) and $\gamma_R(\eta)$ is even since
$\gamma_R(\eta)=\log|\beta_-(\eta)|$.
Moreover, $\gamma_I(\eta)$ is odd because
$\gamma_I(\eta)=\arg\left(-\beta_-(\eta)\right)$ and
$\overline{\beta_-(\eta)}=\beta_-(-\eta)$.
\par
By the definitions, $\beta_-(0)=-e^\k$ and
$\lim_{\eta\to\infty}\gamma_R(\eta)=\infty$.
Since $(d\gamma_R/d\eta)(\eta)=-\Im\mu(\eta)^{-1}>0$ for $\eta>0$,
$|\beta_-(\eta)|=e^{\gamma_R(\eta)}$ is monotone increasing on $(0,\infty)$
and there exists a unique $\eta_*>0$ such that $\gamma_R(\eta_*)=\k+\a$ and
$-\beta_-(\eta_*)=e^{\k+\a+i\theta}$ for a $\theta\in\R$.
By a straightforward computation,  we have $\eta_*=\eta_*(\a)$.
Thus we complete the proof. 
\end{proof}

Finally, we will prove that $g^\pm(\eta)$ decay in $\ell^2_\a$
as $t\to\infty$ if $\eta\ne0$ and that
$\delta(\eta)\simeq -i\lambda_1\eta+\lambda_2\eta^2$.
Let $\delta_R(\eta)=\Re\delta(\eta)$ and $\delta_I(\eta)=\Im\delta(\eta)$.
\begin{lemma}
  \label{cl:delta-eta}
It holds that $\delta_I(\eta)$ is odd, that $\delta_R(\eta)$ is even and
monotone increasing on $[0,\infty)$ and that
$\delta_R(\eta)>\delta(0)=0$ for every $\eta\in\R\setminus\{0\}$.
Moreover,
$(d\delta_I/d\eta)(0)=-\lambda_1$ and $(d^2\delta_R/d\eta^2)(0)=2\lambda_2$.
\end{lemma}
\begin{proof}
We have $\delta(0)=0$ from $w(0)=\cosh\k$ and $\gamma(0)=\k$.
Since $\overline{w(\eta)}=w(-\eta)$ and $\overline{\gamma(\eta)}=\gamma(-\eta)$,
$\Re\delta(\eta)$ is even and $\Im\delta(\eta)$ is odd.

Let $\eta_1=\sup\{\tilde{\eta}>0\mid \Re d\delta/d\eta(\eta)>0\text{ for $\eta\in(0,\tilde{\eta})$}\}$.
Since
\begin{gather*}
\frac{d\delta}{d\eta}= i\frac{\k^{-1}\sinh\k-w(\eta)}{\mu(\eta)}\,,
\quad
\frac{d^2\delta}{d\eta^2}= \frac{\k^{-1}\sinh\k w(\eta)-1}{\mu(\eta)^3}\,,
\end{gather*}
we have $\Re d\delta/d\eta(0)=0$, $d^2\delta/d\eta^2(0)=(\sinh2\k/(2\k)-1)/\sinh^3\k>0$,
and $\eta_1>0$.
Suppose that $\eta_1<\infty$. Then $\Re d\delta/d\eta(\eta_1)=0$ and
\begin{equation*}
  w(\eta_1)-\k^{-1}\sinh\k=\sigma\mu(\eta_1)\,.
\end{equation*}
Squaring both sides and subtracting the right-hand side from the left-hand side, we have
\begin{align}
& \left(w(\eta_1)-\k^{-1}\sinh\k\right)^2-\sigma^2\mu(\eta_1)^2\notag
  \\ =&
\label{eq:rpart}
(1-\sigma^2)(\sinh^2\k-\eta_1^2)+1+\k^{-2}\sinh^2\k-2\k^{-1}\sinh\k\cosh\k
\\ & +2i\eta_1\cosh\k\left(1-\sigma^2-\k^{-1}\tanh\k\right)=0\,, \notag
\end{align}
and $1-\sigma^2=\tanh\k/\k$ since $\eta_1>0$. Substituting $1-\sigma^2=\tanh\k/\k$ into \eqref{eq:rpart},
we have
\begin{align*}
&  (1-\sigma^2)(\sinh^2\k-\eta_1^2)+1+\k^{-2}\sinh^2\k-2\k^{-1}\sinh\k\cosh\k
  \\ < &  \k^{-2}\sinh^2\k+\k^{-1}(\tanh\k\sinh^2\k-2\sinh\k\cosh\k)+1
\\=& \left(\frac{\tanh\k}{\k}-1\right)\left(\frac{\sinh2\k}{2\k}-1\right)<0\,,
\end{align*}
which is a contradiction.
\end{proof}

\bigskip

\section{Daroboux transformations}
\label{sec:Darboux}
By a B\"{a}cklund transformation
\begin{equation}
  \label{eq:BT}
  D_s\tau_n\cdot\tau_n'=\tau_{n+1}\tau_{n-1}'\,,
  \quad   D_x\tau_{n+1}\cdot\tau_n'=-\tau_n\tau_{n+1}'\,,
\end{equation}
$N$-soliton solutions of \eqref{eq:2dToda} are connected to
$(N-1)$-soliton solutions of \eqref{eq:2dToda} (see \cite{Hirota}).
Suppose that $\tau_n$ and $\tau_n'$ satisfy the bilinear
equation \eqref{eq:bilinear}. Let
\begin{gather*}
V_n=\pd_x\pd_s\tau_n\,,\quad q_n=\log\frac{\tau_n}{\tau_{n-1}}\,,
\\
V_n'=\pd_x\pd_s\tau_n'\,, \quad q_n'=\log\frac{\tau_n'}{\tau_{n-1}'}\,.
\end{gather*}
Then \eqref{eq:BT} is translated into
\begin{gather}
  \label{eq:BT1}
  \pd_s(q_n-q_n')=(1-e^{-\pd})e^{q_{n+1}-q_n'}\,,
  \\
  \label{eq:BT2}
  \pd_x(q_{n+1}-q_n')=(1-e^\pd)e^{q_n'-q_n}\,,
\end{gather}
and
\begin{gather*}
  R_n:=\log(1+V_n)=q_{n+1}-q_n\,,\quad
  R_n':=\log(1+V_n')=q_{n+1}'-q_n'\,.
\end{gather*}
Note that \eqref{eq:BT1} and \eqref{eq:BT2} resemble
a B\"{a}cklund tranformation for $1$-dimensional
Toda lattices (\cite{Toda-Wadati}). Let
\begin{gather}
  \label{eq:Miura1a}
  u_n:=e^{q_{n+1}-q_n'}=\frac{\tau_{n+1}\tau_{n-1}'}{\tau_n\tau_n'}
  =\pd_s\log\frac{\tau_n}{\tau_n'}\,,
  \\
  \label{eq:Miura1b}
  v_n:=e^{q_n'-q_n}=\frac{\tau_{n-1}\tau_n'}{\tau_n\tau_{n-1}'}
    =\pd_x\log\frac{\tau_{n-1}'}{\tau_n}\,.
  \end{gather}
Then $(u,v)=\{(u_n,v_n)\}_{n\in\Z}$ is a solution of the modified Toda equation
  \begin{equation*}
\left\{
  \begin{aligned}
    & \pd_x u=u(1-e^\pd)v\,,
    \\ & \pd_sv=v(e^{-\pd}-1)u\,.
  \end{aligned}\right.
  \end{equation*}
  \par
Linearizing \eqref{eq:BT1} and \eqref{eq:BT2}, we obtain
Darboux transformations
\begin{equation}
  \label{eq:Darboux}
  \left\{
\begin{gathered}
  A(\pd_s)\bQ=A'(\pd_s)\bQd\,,
  \quad B(\pd_x)\bQ=B'(\pd_x)\bQd\,,\\
A(\pd_s)=\pd_s-(1-e^{-\pd})ue^\pd\,,\quad
A'(\pd_s)=\pd_s-(1-e^{-\pd})u\,,\\
B(\pd_x)=\pd_x-(1-e^{-\pd})v\,,\quad
B'(\pd_x)=e^{-\pd}\pd_x-(1-e^{-\pd})v\,, 
\end{gathered}\right.
\end{equation}
where $\bQ=\{\mathbf{Q_n}\}_{n\in\Z}$ and $\bQd=\{\mathbf{Q_n'}\}_{n\in \Z}$.
\par

If $\tau_n$ and $\tau_n'$ satisfy \eqref{eq:BT} and $V=V_n$ is a solution
of \eqref{eq:2dToda},
\begin{gather}
\label{eq:Miura2}  
  1+V_n=\frac{\tau_{n+1}\tau_{n-1}}{\tau_n^2}=u_nv_n\,,\\
\label{eq:Miura3}
  1+V_n'=\frac{\tau_{n+1}'\tau_{n-1}'}{(\tau_n')^2}=u_nv_{n+1}\,.
\end{gather}
Let $\Psi=\{\Psi_n\}_{n\in\Z}$ with $\Psi_n=\tau_n'/\tau_n$,
$V=\{V_n\}_{n\in\Z}$, $V'=\{V_n'\}_{n\in\Z}$
and $q=\{q_n\}_{n\in\Z}$, $q'=\{q_n'\}_{n\in\Z}$. Let
\begin{gather*}
  L_1=\pd_s+(1+V)e^{-\pd}\,,\quad
  L_2=\pd_x-e^\pd-e^\pd(\pd_xq)e^{-\pd}\,,
\\
L_1'=\pd_s+(1+V')e^{-\pd}\,,\quad
L_2'=\pd_x-e^\pd- e^\pd(\pd_xq')e^{-\pd}\,.
\end{gather*}
By \eqref{eq:Miura1a}--\eqref{eq:Miura3},
\begin{gather}
\label{eq:Darboux1a}
M_1(\pd_s):=\Psi^{-1}L_1\Psi=\pd_s-u(1-e^{-\pd})\,,
\\ \label{eq:Darboux1b}
M_2(\pd_x):=\Psi^{-1}L_2\Psi=\pd_x-e^\pd v(1-e^{-\pd})\,,
\\ \label{eq:Darboux1'a}
M_1'(\pd_s):=\Psi^{-1}e^{-\pd}L_1'e^\pd\Psi 
=\pd_s-(1-e^{-\pd})u=A'(\pd_s)\,,
\\ \label{eq:Darboux1'b}
M_2'(\pd_x):=\Psi^{-1}e^{-\pd}L_2' e^\pd\Psi 
=\pd_x-(e^\pd-1)v=e^\pd B'(\pd_x)\,,
\\ \label{eq:MM'}
(1-e^{-\pd})M_1=M_1'(1-e^{-\pd})\,,\quad (1-e^{-\pd})M_2=M_2'(1-e^{-\pd})\,,
\end{gather}
and
\begin{equation}
  \label{eq:A}
A(\pd_s)= -e^{-\pd}M_1(\pd_s)^*e^\pd \,,
\quad
B(\pd_x)= -e^{-\pd}M_2(\pd_x)^*e^\pd\,.
\end{equation}
follow from \eqref{eq:Darboux1a} and \eqref{eq:Darboux1b}.
By \eqref{eq:Darboux1a}--\eqref{eq:A}, we have the following.
\begin{claim}
  \label{cl:A,B,dual}
Formally,
  \begin{gather*}
A^*=-(e^\pd-1)^{-1}A'(e^\pd-1)=-\pd_s+e^{-\pd}u(e^\pd-1)\,,
\\   
(A')^*=-(1-e^{-\pd})^{-1}A(1-e^{-\pd})=-\pd_s+u(e^\pd-1)\,,
\\
B^*=-(1-e^{-\pd})^{-1}B'(e^\pd-1)=-\pd_x+v(e^\pd-1)\,,
\\   
(B')^*=-(1-e^{-\pd})^{-1}B(e^\pd-1)=-e^\pd\pd_x+v(e^\pd-1)\,.
\end{gather*}
\end{claim}
\par
Hereafter, let $\tau_n=1$ and $\tau_n'$ be as \eqref{def:taun'}.
Then $V_n=0$, $V_n'=V^\k_n$ and
$\tau_n$ and $\tau_n'$ satisfy
the B\"{a}cklund transformation \eqref{eq:BT}.
 Moreover,
 \begin{gather}
   \label{eq:un}
u_n=-\cosh\k+\sinh\k\tanh\k z_n(t)\,,
  \\ \label{eq:vn}
  v_n =-\cosh\k-\sinh\k\tanh\k z_{n-1}(t)\,.
\end{gather}
The Darboux transformation gives a correspondence between
the solutions of \eqref{eq:linear-1} and those of \eqref{eq:linear-0}.
\begin{lemma}
  \label{lem:Darboux}
Let $\beta_1\in\C$ and $\beta_2\in\C\setminus\{a,1/a\}$. Then
\begin{equation}
  \label{eq:8}
\begin{split}  
A'e^{-\pd}\left\{\Phi(\beta_1)\Phi^*(\beta_2)\right\}
=& Ae^{-\pd}\left\{\Phi^0(\beta_1)\Phi^{0,*}(\beta_2)\right\}
    \\ =&
-(1-e^{-\pd})\left\{u\left(e^{-\pd}\Phi(\beta_1)\right)\Phi^{0,*}
(\beta_2)\right\}\,,
  \end{split}  
\end{equation}
\begin{equation}
  \label{eq:11}
\begin{split}  
B'e^{-\pd}\left\{\Phi(\beta_1)\Phi^*(\beta_2)\right\}
=& Be^{-\pd}\left\{\Phi^0(\beta_1)\Phi^{0,*}(\beta_2)\right\}
\\ =&
(1-e^{-\pd})e^{-\pd}\left\{\Phi(\beta_1)\Phi^{0,*}(\beta_2)\right\}\,.
  \end{split}  
\end{equation}
\end{lemma}

\begin{corollary}
  \label{cor:Darboux}
For any $\beta\in\C$,
\begin{gather}
\label{eq:corDarboux-1}
A'\frac{e^{-\pd}\Phi(\beta)}{\tau'}= B'\frac{e^{-\pd}\Phi(\beta)}{\tau'}=0\,,
\\
\label{eq:corDarboux-2}
A^*e^{-\pd}\frac{\Phi^0(\beta)}{\tau'}
= B^*e^{-\pd}\frac{\Phi^0(\beta)}{\tau'}=0\,.    
\end{gather}
\end{corollary}
\begin{proof}
Taking the residue of \eqref{eq:8} and \eqref{eq:11} at $\beta_2=a$, we have
\eqref{eq:corDarboux-1}.
Combining \eqref{eq:corDarboux-1} with Claims~\ref{cl:A,B,dual} and
\eqref{eq:Phi-updown-1}, we have \eqref{eq:corDarboux-2}.
Thus we complete the proof.  
\end{proof}

\begin{proof}[Proof of Lemma~\ref{lem:Darboux}]
By Lemma~\ref{lem:Jost},
  \begin{align*}
    A'e^{-\pd}(\Phi\Phi^*)=&
\pd_se^{-\pd}(\Phi\Phi^*)
 -(1-e^{-\pd})\left\{ue^{-\pd}(\Phi\Phi^*)\right\}
    \\=&
(1-e^{-\pd})\left[(e^{-\pd}\Phi)
\left\{(1+V')\Phi^*-u\left(e^{-\pd}\Phi^*\right)\right\}\right]\,.
  \end{align*}
 By \eqref{eq:Miura3} and \eqref{eq:Phi-updown-2},
 \begin{align*}
 (1+V')\Phi^*-u(e^{-\pd}\Phi^*)
   =& u\left\{(e^\pd v)\Phi^*-e^{-\pd}\Phi^*\right\}
      = -u\Phi^{0,*}\,.
 \end{align*}
Using \eqref{eq:Jost-0}, Lemma~\ref{lem:Jost} and Claim~\ref{cl:Phi-updown},
we can prove the rest in the same way.
\end{proof}

\bigskip

\section{Linear stability of the null solution}
\label{sec:exp-st0}
In this section, we will prove the exponential linear stability of the null solution in
$\ell^2_\a H^1(\R)\times \ell^2_\a L^2(\R)$.
By the standard argument, we can prove that \eqref{eq:linear-0}
is well-posed on $\ell^2_\a H^1(\R) \times \ell^2_\a$.
\begin{lemma}
  \label{lem:iv-0}
  Let $\a\in\R$ and $(\mathbf{R_0},\mathbf{R_1})\in
 \ell^2_\a H^1\times \ell^2_\a L^2$. Then, the initial value problem
 \begin{equation}
  \label{eq:iv-0}
\left\{
  \begin{aligned}
&  (\pd_t^2-\pd_y^2)\bR=(e^\pd-2+e^{-\pd})\bR\,,
\\ &
\bR(0)=\mathbf{R_0}\,,\quad \pd_t\bR(0)=\mathbf{R_1}\,,
\end{aligned}\right.
 \end{equation}
has a unique solution in the class
\begin{equation}
\label{eq:ivc}
C(\R;\ell^2_\a H^1)\cap C^1(\R;\ell^2_\a L^2)\,.
\end{equation}  
\end{lemma}

Now we will estimate the growth rate of solutions to \eqref{eq:linear-0}.
\begin{proposition}
  \label{prop:exp-0}
  Let $\a\in\R\setminus\{0\}$ and $\bR$ be a solution of \eqref{eq:linear-0}
  in the class \eqref{eq:ivc}. Then, for every $t\in\R$,
\begin{equation*}
 \|\bR(t)\|_{\ell^2_\a H^1}+\|\pd_t\bR(t)\|_{\ell^2_\a L^2}
 \le Ke^{2|t\sinh\frac\a2|}
 (\|\bR(0)\|_{\ell^2_\a H^1}+\|\pd_t\bR(0)\|_{\ell^2_\a L^2})\,,
\end{equation*}
where $K$ is a positive constant that depends only on $\a$.
\end{proposition}

To prove Proposition~\ref{prop:exp-0}, we will use the Planchrel theorem.
Let  $\bR^\a(n,y,t):=e^{\a n}\bR(n,y,t)$ and
\begin{equation*}
\widehat{\bR^\a}(\xi,\eta,t)=\frac{1}{2\pi}\sum_{n\in\Z}
  \int_\R \bR(n,y,t)e^{\a n-i(n\xi+y\eta)}\,dy\,.
\end{equation*}
Then
\begin{equation}
  \label{eq:10}
  \begin{split}
  &  \|\bR(t)\|_{\ell^2_\a H^1(\R)}^2+\|\pd_t\bR(t)\|_{\ell^2_\a L^2(\R)}^2
  \\ = &
\left\|\la \eta\ra\widehat{\bR^\a} \right\|_{L^2([-\pi,\pi]\times\R)}^2
 +\left\|\widehat{\pd_t\bR^\a}\right\|_{L^2([-\pi,\pi]\times\R)}^2\,,    
  \end{split}
\end{equation}
where $\la \eta\ra=\sqrt{1+\eta^2}$.
\par

If $\bR$ is a solution of \eqref{eq:linear-0}, 
\begin{align}
\label{eq:wbRa}
 (\pd_t^2+\eta^2)\widehat{\bR^\a}=& (e^{\a-i\xi}-2+e^{-\a+i\xi})\widehat{\bR^\a}  
  \\  =& \notag
-4\sin^2\frac{\xi+i\a}{2}\widehat{\bR^\a}\,,
\end{align}
and
\begin{equation}
  \label{eq:bR}
  \begin{pmatrix}
\omega\widehat{\bR^\a}(t) \\ \pd_t\widehat{\bR^\a}(t)
\end{pmatrix}
=
\begin{pmatrix}
  \cos t\omega & \sin t\omega
  \\ -\sin t\omega & \cos t\omega
\end{pmatrix}
  \begin{pmatrix}
\omega\widehat{\bR^\a}(0) \\ \pd_t\widehat{\bR^\a}(0)
\end{pmatrix}\,,  
\end{equation}
where 
$\omega(\xi,\eta)
=\left(\eta^2+4\sin^2\frac{\xi+i\a}{2}\right)^{1/2}$ and
$\arg\omega\in(-\pi/2,\pi/2]$.
\par
To investigate the growth rate of $\widehat{\bR^\a}(t)$,
we need the following.
\begin{lemma}
  \label{lem:Im-omega}
  Let $\a>0$, $\xi\in[0,\pi]$ and $\eta\in\R$. Then
$\omega\to0$ if and only if $(\xi,\eta)\to
(0,\pm2\sinh(\a/2))$ and
\begin{gather*}
\omega(\xi,\eta)=|\eta|(1+o(1))\quad\text{as $\eta\to\pm\infty$,}\\
\Im \omega(-\xi,\eta)=-\Im\omega(\xi,\eta)\,,
\quad 0\le \Im\omega(\xi,\eta)\le \Im\omega(\xi,0)\le 2\sinh\frac{\a}{2}\,.
  \end{gather*}
\end{lemma}
\begin{proof}
 Since
  \begin{align*}
\omega(\xi,\eta)^2=
\eta^2+2-2\cos\xi\cosh\a+2i\sin\xi\sinh\a=:\zeta_1+i\zeta_2\,,
  \end{align*}
  we have $\omega(\xi,\eta)=|\eta|(1+o(1))$ as $\eta\to\pm\infty$. 
  If $\zeta_2\to0$, then $\xi\to0$ or $\xi\to\pi$.
  Since $\omega(0,\eta)^2=\eta^2-4\sinh^2(\a/2)$ and
  $\omega(\pi,\eta)^2=\eta^2+4\cosh^2(\a/2)$, we have
  $\omega\to0$ if and only if $\xi\to0$ and $\eta\to \pm2\sinh(\a/2)$.
  \par
Let $\omega_R+i\omega_I:=\omega$.
Then $\omega_I>0$ for $\xi\in(0,\pi)$,
$\omega_R>0$ unless $\xi=0$ and $|\eta|\le 2\sinh(\a/2)$, and
  $\omega(0,\eta)=i\sqrt{4\sinh^2(\a/2)-\eta^2}$
  if $|\eta|\le 2\sinh(\a/2)$.
Since $\zeta_1$ is even and $\zeta_2$ is odd in $\xi$,
$\omega_R$ is even and $\omega_I$ is odd in $\xi\in(-\pi,\pi)$.
\par
Since $\omega=(\zeta_1+i\zeta_2)^{1/2}$ and
$\pd(\zeta_1+i\zeta_2)/\pd\eta=2\eta$,
we have $\pd_{\zeta_1}\omega_I=-\omega_I/2|\omega|^2$ and
\begin{align*}
\pd_\eta\omega_I=& \pd_{\eta}\zeta_1\pd_{\gamma_1}\omega_I
= -\frac{\eta\omega_I}{|\omega|^2}\,.
\end{align*}
Combining the above with the fact that $\omega_I>0$ for $\xi\in(0,\pi)$,
we have for $\xi\in(0,\pi)$,
\begin{align*}
0\le  \omega_I(\xi,\eta)\le & \omega_I(\xi,0)
=2\cos\frac{\xi}{2}\sinh\frac{\a}{2}\,.
\end{align*}
Thus we complete the proof.
\end{proof}
Now, we are in a position to prove Proposition~\ref{prop:exp-0}.
\begin{proof}[Proof of Proposition~\ref{prop:exp-0}]
Without loss of generality, we may assume $t\ge0$ and $\a>0$.
By Lemma~\ref{lem:Im-omega}, $|\omega_I|\le 2\sinh(\a/2)$ and
there exists a  constant $C_1\ge1$
such that $|\omega(\xi,\eta)|\le C_1\la \eta\ra$
for $\xi\in[-\pi,\pi]$ and $\eta\in\R$.
Hence it follows from \eqref{eq:bR} that
 \begin{align*}
& \left\|\omega\widehat{\bR^\a}(t)\right\|_{L^2([-\pi,\pi]\times\R)}
+\left\|\pd_t\widehat{\bR^\a}(t)\right\|_{L^2([-\pi,\pi]\times\R)}
\\ \le & 2e^{2t\sinh(\a/2)} \left(
 \left\|\omega\widehat{\bR^\a}(0)\right\|_{L^2([-\pi,\pi]\times\R)}
+\left\|\pd_t\widehat{\bR^\a}(0)\right\|_{L^2([-\pi,\pi]\times\R)}
\right)
\\ \le & 2C_1e^{2t\sinh(\a/2)} \left(
 \left\|\la\eta\ra\widehat{\bR^\a}(0)\right\|_{L^2([-\pi,\pi]\times\R)}
+\left\|\pd_t\widehat{\bR^\a}(0)\right\|_{L^2([-\pi,\pi]\times\R)}
\right) \,.
 \end{align*}
 Let $\eps\in(0,2\sinh(\a/2))$ and
 $E_\delta=\{(\xi,\eta)\in\R^2\mid
 |\xi|+|\eta\pm2\sinh(\a/2)|<\delta\}$.
Since $\omega(0,\pm2\sinh(\a/2))=0$, there exists a $\delta>0$  such that
 $|\omega(\xi,\eta)|<\eps$ for $(\xi,\eta)\in E_\delta$ and
  \begin{equation*}
   \left| \frac{\sin t\omega(\xi,\eta)}{\omega(\xi,\eta)}
   \right|\le te^{\eps t}\quad\text{for $(\xi,\eta)\in E_\delta$.}
 \end{equation*}
Thus we have
\begin{align*}
 \left\|\widehat{\bR^\a}(t)\right\|_{L^2(E_\delta)}
  \le &       
 \left\|\cos(t\omega)\widehat{\bR^\a}(0)\right\|_{L^2(E_\delta)}
        + \left\|\frac{\sin(t\omega)}{\omega}
        \pd_t\widehat{\bR^\a}(0)\right\|_{L^2(E_\delta)}
  \\ \le & e^{2t\sinh(\a/2)}
\left\|\la\eta\ra\widehat{\bR^\a}(0)\right\|_{L^2(E_\delta)}
+te^{\eps t}
\left\|\pd_t\widehat{\bR^\a}(0)\right\|_{L^2(E_\delta)}\,.  
\end{align*}
On the other hand, it follows from Lemma~\ref{lem:Im-omega} that
there exists a positive constant $C_2$ such that for
$(\xi,\eta)\in [-\pi,\pi]\times \R\setminus E_\delta$,
$\la \eta\ra\le C_2|\omega(\xi,\eta)|$  and
\begin{align*}
\left\|\la\eta\ra\widehat{\bR^\a}(t)
\right\|_{L^2([-\pi,\pi]\times\R\setminus E_\delta)}  
\le C_2\left\|\omega(\xi,\eta)\widehat{\bR^\a}(t)
\right\|_{L^2([-\pi,\pi]\times\R\setminus E_\delta)}\,.
\end{align*}
This completes the proof of Proposition~\ref{prop:exp-0}.
\end{proof}

As a byproduct of the proof of Proposition~\ref{prop:exp-0},
we have the following.
\begin{corollary}
  \label{cor:exp-0}
  Let $\a>0$, $\eta$, $t_0\in\R$ and $\bR_\eta(t)$ be a solution of
  \begin{equation}
    \label{eq:13}
  (\pd_t^2+\eta^2)\bR_\eta=(e^\pd-2+e^{-\pd})\bR_\eta
  \end{equation}
  in the class $C(\R;\ell^2_\a)$. Then for every $t\ge t_0$,
  \begin{align*}
&    \|\la \eta\ra\bR_\eta(t)\|_{\ell^2_\a}+\|\pd_t\bR_\eta(t)\|_{\ell^2_\a}
 \\   \le & K e^{2(t-t_0)\sinh(\a/2)}
\left(\|\la \eta\ra\bR_\eta(t_0)\|_{\ell^2_\a}+\|\pd_t\bR_\eta(t_0)\|_{\ell^2_\a}
\right)\,,
 \end{align*}
where $K$ is a positive constant that depends only on $\a$.
\end{corollary}
\begin{proof}
Let
  \begin{equation*}
    \widehat{\bR^\a}(\xi,\eta,t)=\frac{1}{\sqrt{2\pi}}
\sum_{n\in\Z}e^{\a n}\bR_\eta(n,t)e^{-in\xi}\,.
  \end{equation*}
Then $\bR^\a$ is a solution of \eqref{eq:wbRa}. By Perseval's identity,
\begin{equation}
  \label{eq:Plancherel2}
  \|\bR_\eta(\cdot,t)\|_{\ell^2_\a}=
  \left\|\widehat{\bR^\a}(\cdot,\eta,t)\right\|_{L^2(-\pi,\pi)}
\end{equation}
In view of the proof of Proposition~\ref{prop:exp-0}, we have
  for $(\xi,\eta)\in[-\pi,\pi]\times \R$ and $t\ge t_0$,
  \begin{align*}
& \left|\la\eta\ra\widehat{\bR^\a}(\xi,\eta,t)\right|
+\left|\pd_t\widehat{\bR^\a}(\xi,\eta,t)\right|
\\ \le & K e^{2(t-t_0)\sinh(\a/2)}
\left(\left|\la\eta\ra\widehat{\bR^\a}(\xi,\eta,t_0)\right|
+\left|\pd_t\widehat{\bR^\a}(\xi,\eta,t_0)\right|\right)\,,
  \end{align*}
  where $K$ is a positive constant that depends only on $\a$.
  Combining the above with \eqref{eq:Plancherel2}, we have
  Corollary~\ref{cor:exp-0}. Thus we complete the proof.
\end{proof}
\bigskip

\section{ Daroboux transformations for the linearized
  equation around $1$-soliton}
\label{sec:Darboux-2}

It is clear that \eqref{eq:6} is well-posed
on $\ell^2_\a H^1(\R)\times \ell^2_\a L^2(\R)$.
\begin{lemma}
  \label{lem:iv-1}
  Let $\a\in\R$ and $(\mathbf{Q_0'},\mathbf{Q_1'})\in
 \ell^2_\a H^1\times \ell^2_\a L^2$. Then, the initial value problem
 \begin{equation*}
\left\{
  \begin{aligned}
&  (\pd_t^2-\pd_y^2)\bQd=(1-e^{-\pd})(1+V^\k)(e^\pd-1)\bQd\,,
\\ &
\bQd(0)=\mathbf{Q_0'}\,,\quad \pd_t\bQd(0)=\mathbf{Q_1'}\,,
\end{aligned}\right.
 \end{equation*}
has a unique solution in the class \eqref{eq:ivc}.
\end{lemma}
Hereafter, we will investigate \eqref{eq:6} instead of \eqref{eq:linear-1}
since $e^\pd-1$ is isomorphic on $\ell^2_\a$ for any $\a>0$
(see Lemma~\ref{lem:pd-1}).
\par
Let $\bR$ be a solution of \eqref{eq:linear-0} in the class \eqref{eq:ivc}
and let $\bQ=(e^{\pd}-1)^{-1}\bR$. Then $\bQ$ is a solution of
\begin{equation}
  \label{eq:7}
(\pd_t^2-\pd_y^2)\bQ=(e^\pd-2+e^{-\pd})\bQ
\end{equation}
in the class \eqref{eq:ivc}.
Taking the Fourier transformation of \eqref{eq:7}
with respect to $y$, we have
\begin{gather}
  \label{eq:linear0-F}
(\pd_t^2+\eta^2)\widehat{\bQ}(\eta)=(e^\pd-2+e^{-\pd})\widehat{\bQ}(\eta)\,.
\end{gather}
\par 

Next, we will confirm that a Darboux transformation \eqref{eq:Darboux}
gives a correspondence between solutions of \eqref{eq:6}
and solutions of \eqref{eq:7}.
By \eqref{eq:v-change} and \eqref{eq:Darboux},
\begin{equation}
  \label{eq:4}
  \begin{pmatrix}
   A(-\pd_y) & 1 \\ B(\pd_y) & 1
 \end{pmatrix}
 \begin{pmatrix}
   \bQ \\ \pd_t\bQ
 \end{pmatrix}
 =
  \begin{pmatrix}
   A'(-\pd_y) & 1 \\ B'(\pd_y) & e^{-\pd}
 \end{pmatrix}
 \begin{pmatrix}
   \bQd \\ \pd_t\bQd
 \end{pmatrix}
\end{equation}
Multiplying \eqref{eq:4} by $\begin{pmatrix}
  1 & -1\\ -1 &e^\pd
\end{pmatrix}$ from the left side, we have
\begin{multline}
  \label{eq:Darboux2}
  \begin{pmatrix}
 A(-\pd_y)-B(\pd_y) & 0 \\ e^\pd B(\pd_y)-A(-\pd_y) & e^\pd-1
 \end{pmatrix}
 \begin{pmatrix}
   \bQ \\ \pd_t\bQ
 \end{pmatrix}
\\ =
  \begin{pmatrix}
   A'(-\pd_y) -B'(\pd_y)& 1-e^{-\pd} \\ e^\pd B'(\pd_y)-A'(-\pd_y) & 0
 \end{pmatrix}
 \begin{pmatrix}
   \bQd \\ \pd_t\bQd
 \end{pmatrix}\,.  
\end{multline}
\par
Let
$\mathbf{Q_1}=\widehat{\bQ}$, $\mathbf{Q_2}=\pd_t\widehat{\bQ}$ and
$\mathbf{Q'_1}=\widehat{\bQd}$, $\mathbf{Q'_2}=\pd_t\widehat{\bQd}$.
Then
\begin{multline}
  \label{eq:Darboux2-F}
  \begin{pmatrix}
 C(\eta) & 0 \\ e^\pd B(i\eta)-A(-i\eta) & e^\pd-1
 \end{pmatrix}
 \begin{pmatrix}
   \mathbf{Q_1}(\eta) \\ \mathbf{Q_2}(\eta)
 \end{pmatrix}
\\ =
  \begin{pmatrix}
   A'(-i\eta) -B'(i\eta)& 1-e^{-\pd} \\ C'(\eta) & 0
 \end{pmatrix}
 \begin{pmatrix}
\mathbf{Q'_1}(\eta) \\ \mathbf{Q'_2}(\eta)
 \end{pmatrix}\,,
\end{multline}
where
$C(\eta)= A(-i\eta)-B(i\eta)$ and $C'(\eta)=e^\pd B'(i\eta)-A'(-i\eta)$.
Note that $V^\k$ is independent of $y$.
\par
We will show that for every $\eta\in\R$, a 
Darboux transformation \eqref{eq:Darboux2-F} gives a correspondence
between solutions of \eqref{eq:linear0-F} and those of \eqref{eq:linear1-F}.
\begin{lemma}
  \label{lem:correspondence}
Let $t_0\in\R$ and $\eta\in\R$.   Assume that $\a\in\R\setminus\{0\}$ and
that $\widehat{\bQd}(\eta)$ and $\widehat{\bQ}(\eta)$ are solutions
of \eqref{eq:linear1-F} and \eqref{eq:linear0-F}
in the class $C^2(\R;\ell^2_\a)$, respectivey.
Let $(\mathbf{Q_1}(\eta),\mathbf{Q_2}(\eta))
=(\widehat{\bQ}(\eta),\pd_t\widehat{\bQ}(\eta))$ and
$(\mathbf{Q_1'}(\eta),\mathbf{Q_2'}(\eta))
=(\widehat{\bQd}(\eta),\pd_t\widehat{\bQd}(\eta))$.
If \eqref{eq:Darboux2-F} holds at $t=t_0$, then \eqref{eq:Darboux2-F} holds
for every $t\in\R$.
\end{lemma}
\begin{proof}
  We can prove Lemma~\ref{lem:correspondence} in the same way as
  \cite[Lemma~5]{MP}.
\end{proof}
\par

If $|\eta|$ is small, mappings $C(\eta)$ and $C'(\eta)$ are not
invertible on $\ell^2_\a$.
First, we will investigate the kernel and the cokernel of
$C(\eta)$ and $C'(\eta)$. 

\begin{lemma}
  \label{lem:kerC-C'}
Let $\a\in(0,\k)$ and let $C(\eta)$ and
$C'(\eta)$ be operators on $\ell^2_\a$.
  \begin{itemize}
\item[\rm{(i)}]
If $\eta\in(-\eta_*(\a),\eta_*(\a))$,
\begin{gather*}
  \ker C(\eta)=\{0\}\,,\quad \ker C'(\eta)=\spann\{g^+(\eta)\}
  \\ \ker C(\eta)^*=\spann\{e^{-\pd}\tg^{+,*}(\eta)\}\,, \quad
  \ker C'(\eta)^*=\{0\}\,.  
\end{gather*}
\item[\rm{(ii)}]  
If $\pm\eta>\eta_*(\a)$,
\begin{gather*}
  \ker C(\eta)=\ker C'(\eta)=\{0\}\,,\quad
  \ker C(\eta)^*=\ker C'(\eta)^*=\{0\}\,.
\end{gather*}
\end{itemize}
\end{lemma}
\begin{proof}[Proof of Lemma~\ref{lem:kerC-C'}]
Let $\widetilde{\Psi}=e^{y\cosh\k}\Psi$ so that
$\widetilde{\Psi}$ is independent of $y$.
By \eqref{eq:Darboux1a}--\eqref{eq:A},
\begin{align*}
  A-B=& e^{-\pd}\Psi(L_2-L_1)^*\Psi^{-1}e^\pd\,,
\\  =&  -e^{-\pd}\widetilde{\Psi}(2\pd_y+2\cosh\k+e^\pd+e^{-\pd})\widetilde{\Psi}^{-1}e^\pd\,,
\end{align*}
\begin{align*}
e^\pd B'-A'=&  (1-e^{-\pd})\Psi^{-1}(L_2-L_1)\Psi(1-e^{\pd})^{-1}\,,
  \\= & (1-e^{-\pd})\widetilde{\Psi}^{-1}(2\pd_y-2\cosh\k-e^\pd-e^{-\pd})
        \widetilde{\Psi}(1-e^{\pd})^{-1}\,.
\end{align*}
Let $D(\eta)=e^\pd+e^{-\pd}+2i\eta+2\cosh\k$. Then
\begin{gather*}
C(\eta)= -e^{-\pd}\tau' D(\eta)(\tau')^{-1}e^\pd\,,
  \quad
 C'(\eta)=-(1-e^{-\pd})(\tau')^{-1}D(-\eta)\tau'(1-e^\pd)^{-1}\,.
\end{gather*}

\par
Suppose that $C(\eta)q=0$.
We have $C(\eta)q=0$ if and only if $f=(e^\pd q)/\tau'$ satisfies
\begin{equation}
  \label{eq:kerC1}
  D(\eta)f=0\,.
\end{equation}
Since $\{\beta_+(-\eta)^n,\beta_-(-\eta)^n\}$ is a fundamental system
of \eqref{eq:kerC1}, 
\begin{equation*}
  q_n=\tau'_{n-1}\left\{c_1\beta_+(-\eta)^n+ c_2\beta_-(-\eta)^n\right\}\,,
\end{equation*}
where $c_1$ and $c_2$ are constants.
By \eqref{eq:Miura1b} and \eqref{eq:vn},
\begin{equation}
\label{eq:taun+1/n}
\begin{aligned}
 \frac{\tau_{n+1}'}{\tau_n'}= v_{n+1}
 &  \to
  \begin{cases}
  & -e^\k<-1\quad\text{as $n\to\infty$,}
\\ & -e^{-\k}>-1\quad\text{as $n\to-\infty$.}
  \end{cases}
\end{aligned}
\end{equation}
Combining the above with Claim~\ref{cl:beta1-beta2},
we have $\ker C(\eta)=\{0\}$. We can prove $\ker C'(\eta)^*=\{0\}$
in the same way.
\par

Suppose that $C'(\eta)q'=0$. Then
$f'=\{f'_n\}_{n\in\Z}=\tau'(1-e^{-\pd})^{-1}q'$ satisfies
$$D(-\eta)f'=0\,,$$
and  $f_n'=c_1\beta_+(-\eta)^n+c_2\beta_-(-\eta)^n$,
where $c_1$ and $c_2$ are constants.  By
Claim~\ref{cl:beta1-beta2} and \eqref{eq:taun+1/n},
we have $\ker C'(\eta)=\spann\{g^+(\eta)\}$ if $|\eta|<\eta_*(\a)$
and $\ker C'(\eta)=\{0\}$ if $|\eta|>\eta_*(\a)$.
Similarly, we have
$\ker C(\eta)^*=\spann\{e^{-\pd}\tg^{+,*}\}$ if
$|\eta|<\eta_*(\a)$ and $\ker C(\eta)^*=\{0\}$ if $|\eta|>\eta_*(\a)$.
\end{proof}

Next, we will prove that $C(\eta)$ and $C'(\eta)$ are bijective
if $|\eta|>\eta_*(\a)$.
\begin{lemma}
  \label{lem:CC'-high}
  Let $\a\in(0,2\k)$ and $\eta_0>\eta_*(\a)$. Then, there exists a positive
 constant $K$ such that 
  \begin{equation*}
    \|C(\eta)^{-1}\|_{B(\ell^2_\a)}+\|C'(\eta)^{-1}\|_{B(\ell^2_\a)}
    \le \frac{K}{1+|\eta|}
\quad\text{for $\pm\eta\ge \eta_0$.}    
    \end{equation*}
  \end{lemma}
  \begin{proof}[Proof of Lemma~\ref{lem:CC'-high}]
Since $e^{\pm\pd}$ and $(e^\pd-1)^{-1}$ are bounded on $\ell^2_\a$
by Lemmma~\ref{lem:pd-1},
it suffices to show that for $\pm\eta\ge \eta_0$,
\begin{equation*}
\left\|\tau'D(\eta)^{-1}(\tau')^{-1}\right\|_{B(\ell^2_\a)}
 +
 \left\|(\tau')^{-1}D(-\eta)^{-1}\tau'\right\|_{B(\ell^2_\a)}
=O(\la \eta\ra^{-1})\,.
\end{equation*}
It follows from Claim~\ref{cl:beta1-beta2} that for every
 $\eta_0>\eta_*(\eta)$,  there exists $\eps\in(0,1)$
 such that
 \begin{equation}
   \label{eq:kn-bound}
   |\beta_+(\eta)|=|\beta_-(\eta)|^{-1}\le \eps e^{-(\k+\a)}
   \quad\text{if $\pm\eta\ge \eta_0$.}
 \end{equation}
Thus by Lemma~\ref{lem:fundsol-seq}, $k(\eta)=\{k_n(\eta)\}_{n\in\Z}$  with
 \begin{equation}
 \label{eq:def-kn}
 k_n(\eta)=
\frac{\beta_-(\eta)^{-|n|}}{2\mu(\eta)}
 \end{equation}
 is a Green kernel of $D(\eta)$ on $\ell^2_\a$.
 \par
 
Let $\tau'=\{\tau'_n\}_{n\in\Z}$,  $f=\{f_n\}_{n\in\Z}\in \ell^2_\a$,
$u=\{u_n\}_{n\in\Z}$ and
 \begin{equation}
   \label{eq:sol-Du=f}
   u=\tau'k(\eta)*
   \left(\frac{f}{\tau'}\right)\
 \end{equation}
 so that $\tau'D(\eta)(\tau')^{-1}u=f$.
 By \eqref{def:taun'} and the fact that  $\cosh a/\cosh b\le 2e^{|a-b|}$
 for $a$, $b\in\R$, we have $|\tau_n'/\tau_{n-m}'|\le 2e^{\k|m|}$ and
 \begin{align*}
   e^{\pm\a n}|u_n|\le & e^{\pm\a n}\sum_{m\in\Z}
\left| k_m(\eta)\frac{\tau_n'}{\tau_{n-m}'}f_{n-m}\right|
   \\ \le & 2\sum_{m\in\Z} (e^{(\kappa+\a)|m|}|k_m(\eta)|)e^{\pm\a(n-m)}
            |f_{n-m}|\,.
 \end{align*}
 Combining the above with \eqref{eq:kn-bound} and the definition
 of $k(\eta)$, we have
 \begin{align*}
   \|u\|_{\ell^2_{\pm\a}} \le &
2\left(\sum_{m\in\Z}\{e^{(\k+\a)|m|}|k_m(\eta)|\right)
\|f\|_{\ell^2_{\pm\a}}
   \\ \le & \frac{1+\eps}{(1-\eps)|\mu(\eta)|}| \|f\|_{\ell^2_{\pm\a}}\,,
 \end{align*}
 and for $\eta\in\R$ satisfying $|\eta|\ge \eta_0$,
 \begin{equation}
   \label{eq:D(eta)-bound}
\|\tau'D(\eta)^{-1}(\tau')^{-1}\|_{B(\ell^2_{\pm\a})}\le  K'\la \eta\ra^{-1}\,,
\end{equation}
where $K'$ is a constant that does not depend on $\eta$.
We can esimate $(\tau')^{-1}D(-\eta)^{-1}\tau'$
in the same way. Thus we complete the proof.
\end{proof}

\begin{lemma}
  \label{lem:C-low}
Let $\a\in(0,2\k)$, $\eta_0\in(0,\eta_*(\a))$ and $\eta\in[-\eta_0,\eta_0]$.
If $f\in\ell^2_\a$ and
\begin{equation}
  \label{eq:orth-1}
  \la f, e^{-\pd}\tg^{+,*}(\eta)\ra
  =\sum_{n\in\Z} f_n\overline{\tg^{+,*}_{n-1}(\eta)}=0\,,
\end{equation}
then $C(\eta)u=f$ has a solution satisfying
\begin{equation*}
 \|u\|_{\ell^2_\a}\le K \|f\|_{\ell^2_\a}\,,
\end{equation*}
where $K$ is a positive constant that depends only on $\a$ and $\eta_0$.
\end{lemma}
\begin{proof}[Proof of Lemma~\ref{lem:C-low}]
Let 
\begin{align*}
u:= -e^{-\pd}\tau'k(\eta)* \left(\frac{1}{\tau'}e^\pd f\right)\,,
\end{align*}
and $u=\{u_n\}_{n\in\Z}$. Then
\begin{equation*}
  u_n=-\sum_{m\in\Z}k_{n-m}(\eta)\frac{\tau_{n-1}'}{\tau_{m-1}'}f_m\,,
\end{equation*}
We formally have $C(\eta)u=f$ from \eqref{eq:sol-Du=f}
since $C(\eta)=-e^{-\pd}\tau'D(\eta)(\tau')^{-1}e^\pd$.
\par
By Claim~\ref{cl:beta1-beta2} and \eqref{eq:def-kn},
\begin{equation}
  \label{eq:kn-m}
  \sup_{\eta\in\R} |k_{n-m}(\eta)|\lesssim e^{-\k|n-m|}\,.
\end{equation}
By \eqref{def:taun'}, we have for every $m$, $n\in\Z$,
\begin{equation}
  \label{eq:taun/taum-1}
\left|\frac{\tau_{n-1}'}{\tau_{m-1}'}\right|\le 2e^{|m-n|\k}\,.
\end{equation}
If $m\ge n\ge 1+t\sinh \k/\k$ or $m\le n\le 1+t\sinh\k/\k$,
\begin{equation}
  \label{eq:taun/taum-2}
\left|\frac{\tau_{n-1}'}{\tau_{m-1}'}\right|\le 2e^{-|m-n|\k}\,.
\end{equation}
Combining \eqref{eq:kn-m} and \eqref{eq:taun/taum-1}, we have
\begin{align*}
  \left\|\sum_{m\ge n}k_{n-m}(\eta)\frac{\tau_{n-1}'}{\tau_{m-1}'}
  f_m\right\|_{\ell^2_\a}
\lesssim &
 \left\| \sum_{m\ge n}e^{\a n}|f_m| \right\|_{\ell^2}
\\ =&
\left\|\{e^{\a n}\}_{n\le0}*\{e^{\a n}f_n\}\right\|_{\ell^2}
\lesssim \|f\|_{\ell^2_\a}\,.
\end{align*}
By \eqref{eq:kn-m} and \eqref{eq:taun/taum-2},
\begin{multline*}
   \left\|\sum_{m< n}k_{n-m}(\eta)
\frac{\tau_{n-1}'}{\tau_{m-1}'}f_m\right\|_{\ell^2_\a(n\le 1+t\sinh\k/\k)}
 \lesssim 
 \left\| \sum_{m<n}e^{\a n-2\k|n-m|}|f_m| \right\|_{\ell^2}
\\  \qquad  =  \left\|\{e^{(\a-2\k)n}\}_{n\ge0}*\{e^{\a n}f_n\}\right\|_{\ell^2}
\lesssim \|f\|_{\ell^2_\a}\,.
\end{multline*}
\par
Finally, we will estimate
\begin{equation*}
 \left\|\sum_{m< n}k_{n-m}(\eta)
\frac{\tau_{n-1}'}{\tau_{m-1}'}f_m\right\|_{\ell^2_\a(n\ge 1+t\sinh\k/\k)}\,.
\end{equation*}
By \eqref{eq:orth-1},
\begin{align*}
  \la f,e^{-\pd}\tg^{+,*}(\eta)\ra=&
e^{\beta_-(\eta)x-\beta_+(\eta)s}
\sum_{m\in\Z} \frac{f_m\beta_-(\eta)^{m-1}}{\tau_{m-1}'}
=0\,.                       
\end{align*}
Hence, it follows that
\begin{align*}
  \sum_{m<n}k_{n-m}(\eta)\frac{\tau_{n-1}'}{\tau_{m-1}'}f_m=&
\frac{\beta_-(\eta)^{-n+1}\tau_{n-1}'}{2\mu(\eta)}
\sum_{m<n}  \frac{f_m\beta_-(\eta)^{m-1}}{\tau_{m-1}'}
  \\=&
-\frac{\beta_-(\eta)^{-n+1}\tau_{n-1}'}{2\mu(\eta)}
 \sum_{m\ge n} \frac{f_m\beta_-(\eta)^{m-1}}{\tau_{m-1}'}
  \\=& - \sum_{m\ge n}\frac{\beta_-(\eta)^{m-n}}{2\mu(\eta)}
       \frac{\tau_{n-1}'}{\tau_{m-1}'}f_m\,.
\end{align*}
By the assumption, there exists an $\eps\in(0,1)$ such that
$|\beta_-(\eta)|\le \eps e^{(\k+\a)}$ for $\eta\in[-\eta_0,\eta_0]$.
If $m\ge n\ge1+t\sinh\k/$ and $t\ge0$, we have
$|\tau_{n-1}'/\tau_{m-1}'|\le 2e^{(n-m)\k}$ and
\begin{align*}
e^{\a n}\left|\sum_{m<n}k_{n-m}(\eta)\frac{\tau_{n-1}'}{\tau_{m-1}'}f_m\right|
 \lesssim &
\sum_{m\ge n}\eps^{m-n}e^{\a m}|f_m|\,.
\end{align*}
Using Young's inequality, we have
\begin{align*}
  \left\|e^{\a n}\sum_{m<n}k_{n-m}(\eta)\frac{\tau_{n-1}'}{\tau_{m-1}'}f_m
  \right\|_{\ell^2(n\ge 1+t\sinh\k t/\k)}
 \lesssim &
\left\|\sum_{m\ge n}\eps^{m-n}e^{\a m}|f_m|\right\|_{\ell^2}
  \\\lesssim & \frac{1}{1-\eps}\|f\|_{\ell^2_\a}\,. 
\end{align*}
Thus we prove Lemma~\ref{lem:C-low}.
\end{proof}

For $\mathbf{q}=\{(q_1(n),q_2(n))^T\}_{n\in\Z}$ and
$\mathbf{q^*}=\{(q_1^*(n),q_2^*(n))^T\}_{n\in\Z}$,
let
\begin{equation*}
\left\la \mathbf{q},\mathbf{q^*}\right\ra
  =\sum_{n\in\Z} \left\{q_1(n)\overline{q_1^*(n)}+
    q_2(n)\overline{q_2^*(n)}\right\}\,.
\end{equation*}
If $\eta$ is close to $0$, then \eqref{eq:Darboux2-F} is solvable in
$(\mathbf{Q_1}(\eta),\mathbf{Q_2}(\eta))$ under an orthogonality codition on
$(\mathbf{Q_1'}(\eta),\mathbf{Q_2'}(\eta))$.
\begin{lemma}
  \label{lem:sol-Q}
Assume that $\a\in(0,2\k)$ and that $|\eta|\le \eta_0<\eta_*(\a)$. If
\begin{equation}
  \label{eq:orth-Q1'Q2'}
  \left\la
\begin{pmatrix}
  \mathbf{Q_1'}(\eta) \\ \mathbf{Q_2'}(\eta)
\end{pmatrix},
\begin{pmatrix}
 \pd_tg^{+,*}(\eta) \\ -g^{+,*}(\eta)  
\end{pmatrix}
\right\ra=0\,,
\end{equation}
then there exists a unique $(\mathbf{Q_1}(\eta), \mathbf{Q_2}(\eta))$
satisfying \eqref{eq:Darboux2-F} and
\begin{equation}
\label{eq:lem-solQ}
\|\mathbf{Q_1}(\eta)\|_{\ell^2_\a}+\|\mathbf{Q_2}(\eta)\|_{\ell^2_\a}
\le K\left(\|\mathbf{Q_1'}(\eta)\|_{\ell^2_\a}
  +\|\mathbf{Q_2'}(\eta)\|_{\ell^2_\a}\right)\,,
\end{equation}
where $K$ is a positive constant $K$ that depends only on $\a$ and $\eta_0$.
\end{lemma}
To prove Lemma~\ref{lem:sol-Q}, we need the following.
\begin{claim}
  \label{cl:pd-t,AB}
  \begin{gather}
  \label{eq:clpd-t,AB-1}
\left\{A'(-i\eta)-B'(i\eta)\right\}g^+(\eta)+(1-e^{-\pd})\pd_tg^+(\eta)=0\,,
  \\   \label{eq:clpd-t,AB-2}
\pd_tg^{+,*}(\eta)=\left\{A'(-i\eta)-B'(i\eta)\right\}^*e^{-\pd}\tg^{+,*}(\eta)\,,
\end{gather}
\begin{multline}
   \label{eq:clpd-t,AB-3}
  \left\{A'(-i\eta)-B'(i\eta)\right\}g^-(\eta)
  +(1-e^{-\pd})\pd_tg^-(\eta)
 = -2i\eta (1-e^{-2\pd}) \tg^-(\eta)\,,  
\end{multline}
\begin{multline}
   \label{eq:clpd-t,AB-4}
\pd_tg^{-,*}(\eta)=\left\{A'(-i\eta)-B'(i\eta)\right\}^*e^{-\pd}\tg^{-,*}(\eta)
 -2i\eta(1+e^{-\pd})\tg^{-,*}(\eta)\,.  
\end{multline}
\end{claim}
\begin{lemma}
  \label{lem:orth-relation}
Let $\a\in(0,2\k)$ and $\eta\in(-\eta_*(\a),\eta_*(\a))$. Then
  \begin{gather*}
    \la \tg^\pm(\eta),\pd_tg^{\pm,*}(\eta)\ra
    =\la\pd_t\tg^\pm(\eta),g^{\pm,*}(\eta)\ra\,,
\\
\la g^\pm(\eta),\pd_tg^{\mp,*}(\eta)\ra-\la\pd_tg^\pm(\eta),g^{\mp,*}(\eta)\ra
=-2\mu(\mp\eta)\,.
\end{gather*}
 \end{lemma}
 We will prove Claim~\ref{cl:pd-t,AB} and Lemma~ \ref{lem:orth-relation}
 in Appendix~\ref{sec:appendix2}.

\begin{proof}[Proof of Lemma~\ref{lem:sol-Q}]
  By Lemmas~\ref{lem:kerC-C'} and \ref{lem:C-low},
  Eq.~\eqref{eq:Darboux2-F} is solvable in
  $(\mathbf{Q_1}(\eta), \mathbf{Q_2}(\eta))$ if and only if
  \begin{align*}
&\left \la \left\{A'(-i\eta)-B'(i\eta)\right\}\mathbf{Q_1'}(\eta)
    +(1-e^{-\pd})\mathbf{Q_2'}(\eta), e^{-\pd}\tg^{+,*}(\eta)\right\ra
\\=&
   \left \la
   \begin{pmatrix}
\mathbf{Q_1'}(\eta) \\ \mathbf{Q_2'}(\eta)
   \end{pmatrix},
    \begin{pmatrix}
\left\{A'(-i\eta)-B'(i\eta)\right\}^*e^{-\pd}\tg^{+,*}(\eta)
\\ -(e^\pd-1) e^{-\pd}\tg^{+,*}(\eta)
\end{pmatrix}
\right\ra
 = 0\,.
  \end{align*}
By Claim~\ref{cl:pd-t,AB},
  \begin{align*}
& \left \la \left\{A'(-i\eta)-B'(i\eta)\right\}\mathbf{Q_1'}(\eta)
    +(1-e^{-\pd})\mathbf{Q_2'}(\eta), e^{-\pd}\tg^{+,*}(\eta)\right\ra
\\=&\left \la
\begin{pmatrix}
\mathbf{Q_1'}(\eta) \\ \mathbf{Q_2'}(\eta)      
\end{pmatrix},
    \begin{pmatrix}
\pd_tg^{+,*}(\eta) \\ -g^{+,*}(\eta)      
    \end{pmatrix}\right\ra\,.
  \end{align*}
Hence it follows from Lemma~\ref{lem:C-low} that
for $(\mathbf{Q_1'}(\eta), \mathbf{Q_2'}(\eta))$ satisfying
\eqref{eq:orth-Q1'Q2'}, there exists a unique 
$(\mathbf{Q_1}(\eta), \mathbf{Q_2}(\eta))$ satisfying
\eqref{eq:Darboux2-F} and \eqref{eq:lem-solQ}.
\end{proof}

\begin{lemma}
  \label{lem:C'-low}
Assume that $\a\in(0,2\k)$ and that $|\eta|\le \eta_0<\eta_*(\a)$.
For every $(\mathbf{Q_1}(\eta),\mathbf{Q_2}(\eta))\in \ell^2_\a\times\ell^2_\a$,
there exists a unique $(\mathbf{Q_1'}(\eta),\mathbf{Q_2'}(\eta))$
satisfying \eqref{eq:Darboux2-F} and
\begin{equation}
  \label{eq:orth-Q1'Q2'-new}
  \left\la
\begin{pmatrix}
  \mathbf{Q_1'}(\eta) \\ \mathbf{Q_2'}(\eta)
\end{pmatrix},
\begin{pmatrix}
 \pd_tg^{-,*}(\eta) \\ -g^{-,*}(\eta)  
\end{pmatrix}
\right\ra=0\,,  
\end{equation}
\begin{equation*}
\|\mathbf{Q_1'}(\eta)\|_{\ell^2_\a}+\|\mathbf{Q_2'}(\eta)\|_{\ell^2_\a}
\le K\left(\|\mathbf{Q_1}(\eta)\|_{\ell^2_\a}
  +\|\mathbf{Q_2}(\eta)\|_{\ell^2_\a}\right)\,,
\end{equation*}
where $K$ is a positive constant $K$ that depends only on $\a$ and $\eta_0$.
\end{lemma}
\begin{proof}[Proof of Lemma~\ref{lem:C'-low}]
First, we will prove the uniqueness.
Suppose that
\begin{equation}
    \label{eq:2}
  \begin{pmatrix}
    A'(-i\eta)-B'(i\eta) & 1-e^{-\pd}  \\ C'(\eta) & 0
  \end{pmatrix}
  \begin{pmatrix}
    \mathbf{Q_1'} \\ \mathbf{Q_2'}
  \end{pmatrix}
  =  \begin{pmatrix} 0 \\ 0  \end{pmatrix}\,.
\end{equation}
Then by Lemm~\ref{lem:kerC-C'},
\begin{equation}
  \label{eq:3}
\mathbf{Q_1'}=cg^+(\eta)
=ce^{-\pd}\left(\Phi(\beta_+(-\eta))/e^\pd\tau'\right)  
\end{equation}
for a $c\in\C$.
By Claim~\ref{cl:pd-t,AB}, \eqref{eq:2} and \eqref{eq:3},
\begin{align*}
  \mathbf{Q_2'}=&
-(1-e^{-\pd})^{-1} \left\{A'(-i\eta)-B'(i\eta)\right\} \mathbf{Q_1'}
= c\pd_tg^+(\eta)\,.
\end{align*}
That is,
\begin{equation}
  \label{eq:12}
  \ker\begin{pmatrix}
    A'(-i\eta)-B'(i\eta) & 1-e^{-\pd}  \\ C'(\eta) & 0
  \end{pmatrix}
  =\spann\left\{
    \begin{pmatrix}
      g^+(\eta)\\ \pd_tg^+(\eta)
    \end{pmatrix}\right\}\,.
\end{equation}
If $(\mathbf{Q'_1},\mathbf{Q'_2})$ satisfies \eqref{eq:orth-Q1'Q2'-new}
and \eqref{eq:2},
then $\mathbf{Q'_1}=\mathbf{Q'_2}=0$ by Lemma~\ref{lem:orth-relation}.
Thus we prove the uniqueness of $(\mathbf{Q_1'},\mathbf{Q_2'})$ satisfying \eqref{eq:Darboux2-F} and \eqref{eq:orth-Q1'Q2'-new}.
\par

 Next, we will show that $C'(\eta):\ell^2_\a\to\ell^2_\a$ is surjective.
Following the line of the proof of Lemma~\ref{lem:C-low}, we have
$\operatorname{Range}C'(\eta)^*=\{v\in\ell^2_{-\a}\mid \la g^+(\eta),v\ra=0\}$
and $\operatorname{Range}C'(\eta)^*$ is closed.
Thus by Lemma~\ref{lem:kerC-C'}, $C'(\eta)^*$ as well as $C'(\eta)=C'(\eta)^{**}$
is Fredholm because $C'(\eta)$ is bounded and $\ell^2_\a$ is reflexive.
Since $\ker(C'(\eta)^*)=\{0\}$ and $\ker(C'(\eta))=\spann\{g^+(\eta)\}$ by Lemma~\ref{lem:kerC-C'},
we have $\operatorname{Range}(C'(\eta))=\ell^2_\a$
and $\mathcal{R}':=\operatorname{Range}(C'(\eta)^*)=\spann\{g^+(\eta)\}^\perp$.
Moreover, $C'(\eta)^*:\ell^2_{-\a}\to \mathcal{R}'$ has a bounded inverse,
and we can prove that
$\sup_{\eta\in[-\eta_0,\eta_0]}\|(C'(\eta)^*)^{-1}\|_{B(\mathcal{R}',\ell^2_{-\a})}<\infty$ in the same way
as Lemma~\ref{lem:C-low}.
\par
Next, we will construct an inverse mapping of $C'(\eta)$.  Let
$\delta=\{\delta_{0n}\}_{n\in\Z}$, where $\delta_{0n}$ is the
Kronecker delta, and
$\mathcal{D}=\{u=\{u_n\}_{n\in\Z}\in\ell^2_\a \mid u_0=0\}$.  Since
$\la g^+(\eta),\delta\ra=g^+_0(\eta)\ne0$, we have
$\ell^2_{-\a}=\spann\{\delta\}\oplus
\operatorname{Range}(C'(\eta)^*)$ and
$\ker C'(\eta)\cap \mathcal{D}=\{0\}$.
\par
Suppose that $u\in\mathcal{D}$ and that $\varphi\in \ell^2_{-\a}$.
Then $\varphi=c\delta+f^*$ with $f^*\in\operatorname{Range}C'(\eta)^*$ and
$c=\la g^+(\eta),\varphi\ra/g^+_0(\eta)$, and
\begin{align*}
  \left|\la u,\varphi\ra\right|=&  \left|\la u,f^*\ra\right|
 \le  K_1 \|C'(\eta)u\|_{\ell^2_\a}\|f^*\|_{\ell^2_{-\a}}\,,
\end{align*}
where $K_1=\sup_{\eta\in[-\eta_0,\eta_0]}
\left\|\left(C'(\eta)^*\right)^{-1}\right\|_{B(\mathcal{R}',\ell^2_{-\a})}$.
Combining the above, we have for every $u\in\mathcal{D}(\eta)$ and $\varphi\in\ell^2_{-a}$,
\begin{align*}
\left|\la u,\varphi\ra\right|
  \le & K_1\|C'(\eta)u\|_{\ell^2_\a}
        \left(\|\varphi\|_{\ell^2_{-\a}}+|c|\right)
 \\ \le & K_1\left(1+\sup_{\eta\in[-\eta_0,\eta_0]}
\frac{\|g^+(\eta)\|_{\ell^2_{-\a}}}{|g^+_0(\eta)|}\right)
\|\varphi\|_{\ell^2_{-\a}}\,.  
\end{align*}
Thus we prove that a map $C'(\eta)^{-1}: \ell^2_\a\to \mathcal{D}$
is uniformly bounded for $\eta\in[-\eta_0,\eta_0]$.
\par
Let 
\begin{gather*}
\mathbf{Q'_{1,0}}(\eta)=C'(\eta)^{-1}\left\{
\left(e^\pd B(i\eta)-A(-i\eta)\right)\mathbf{Q_1}(\eta)+(e^\pd-1)\mathbf{Q_2}(\eta)\right\}\,,
\\
\mathbf{Q'_{2,0}}(\eta)=(1-e^{-\pd})^{-1}\left\{C(\eta)\mathbf{Q_1}(\eta)
+\left(B'(i\eta)-A'(-i\eta)\right)\mathbf{Q_1'}(\eta)\right\}\,,
\\
  \begin{pmatrix}
\mathbf{Q'_1}(\eta) \\\mathbf{Q'_2}(\eta)
\end{pmatrix}
=
\begin{pmatrix}
\mathbf{Q'_{1,0}}(\eta)\\\mathbf{Q'_{2,0}}(\eta)  
\end{pmatrix}
+c
\begin{pmatrix}
  g^+(\eta)\\\pd_tg^-(\eta)
\end{pmatrix}\,.
\end{gather*}
By Lemma~\ref{lem:orth-relation}, 
we can choose $c\in\C$ such that $(\mathbf{Q'_1},\mathbf{Q'_2})$
satisfies \eqref{eq:orth-Q1'Q2'-new}.
By the definitions and \eqref{eq:12},
$(\mathbf{Q'_{1,0}},\mathbf{Q'_{2,0}})$ and $(\mathbf{Q_1'},\mathbf{Q_2'})^T$ satisfy \eqref{eq:Darboux2-F}.
Moreover,
\begin{equation*}
\|\mathbf{Q'_1}(\eta)\|_{\ell^2_\a}+\|\mathbf{Q'_2}(\eta)\|_{\ell^2_\a}\le
K\left(  \|\mathbf{Q_1}(\eta)\|_{\ell^2_\a}
  +\|\mathbf{Q_2}(\eta)\|_{\ell^2_\a}\right)\,,
\end{equation*}
where $K$ is a constant that depends only on $\eta_0$ and $\a$.
This completes the proof of Lemma~\ref{lem:C'-low}.
\end{proof}

\bigskip

\section{Proof of Theorems~\ref{thm:main} and \ref{thm:profile}}
\label{sec:proof}
In this section, we will prove Theorem~\ref{thm:main} by using
the boundedness of Darboux transformations and investigate the asymptotic profile
of solutions to \eqref{eq:linear-1} as $t\to\infty$ (Theorem~\ref{thm:profile}).

To begin with, we will introduce a projection to a subspace spanned by
$g^\pm(\eta)$. Let 
\begin{gather*}
\mathbf{g^\pm}(t,\eta)=
\begin{pmatrix}  g^\pm(\eta) \\ \pd_tg^\pm(\eta)\end{pmatrix}\,,
\quad
\mathbf{g^{\pm,*}}(t,\eta)=
\begin{pmatrix}  \pd_tg^{\pm,*}(\eta) \\ -g^{\pm,*}(\eta)\end{pmatrix}\,,
\\
g^1(t,\eta)=e^{-iy\eta}\left\{g^+(\eta)-i\eta\csch\k\, g^-(\eta)\right\}\,,
\\
g^2(t,\eta)=e^{-iy\eta}\left\{\frac{1}{i\eta}g^+(\eta)+\csch\k\, g^-(\eta)\right\}\,,
\\
g^{1,*}(t,\eta)=e^{-iy\eta}\left\{\frac{1}{i\eta}g^{+,*}(\eta)
  +\csch\,\k g^{-,*}(\eta)\right\}\,,
\\
g^{2,*}(t,\eta)=e^{-iy\eta}\left\{g^{+,*}(\eta)
  -i\eta\csch\k\, g^{-,*}(\eta)\right\}\,,
\end{gather*}
and for $j=1$, $2$, let
\begin{gather*}
\tg^j(t,\eta)=(1-e^{-\pd})^{-1}g^j(t,\eta)\,,\quad
\tg^{j,*}(t,\eta)=(1-e^{-\pd})^{-1}g^{j,*}(t,\eta)\,,
\\
\mathbf{g^j}(t,\eta)=
\begin{pmatrix}  g^j(t,\eta) \\ \pd_tg^j(t,\eta)\end{pmatrix}\,,
\quad
\mathbf{g^{j,*}}(t,\eta)=\begin{pmatrix}  \pd_tg^{j,*}(t,\eta) \\ -g^{j,*}(t,\eta)\end{pmatrix}\,.
\end{gather*}
By \eqref{eq:tg1form}--\eqref{eq:tg2*form} and Lemma~\ref{cl:delta-eta},
\begin{gather}
  \label{eq:tg1}
  \tg^1_n(t,\eta)=e^{-t\delta_R(\eta)-\gamma_R(\eta)z_n(t)}\sech \k z_n(t)
  \cos\{t\delta_I(\eta)+\gamma_I(\eta)z_n(t)\}\,,
  \\ \label{eq:tg2}
  \tg^2_n(t,\eta)=e^{-t\delta_R(\eta)-\gamma_R(\eta)z_n(t)}\sech \k z_n(t)
  \frac{\sin\{t\delta_I(\eta)+\gamma_I(\eta)z_n(t)\}}{\eta}\,,
\\ \label{eq:tg1*}
  \tg^{1,*}_n(t,\eta)=-e^{t\delta_R(\eta)+\gamma_R(\eta)z_n(t)}\sech \k z_n(t)
  \frac{\sin\{t\delta_I(\eta)+\gamma_I(\eta)z_n(t)\}}{\eta}\,,
\\ \label{eq:tg2*}
  \tg^{2,*}_n(t,\eta)=e^{t\delta_R(\eta)+\gamma_R(\eta)z_n(t)}\sech \k z_n(t)
  \cos\{t\delta_I(\eta)+\gamma_I(\eta)z_n(t)\}\,.
\end{gather}
Then $\mathbf{g^j}$ and $\mathbf{g^{j,*}}$ are real valued and
and for $j=1$ and $2$,
\begin{gather*}
\mathbf{g^j}(t,-\eta)=\mathbf{g^j}(t,\eta)\,,
  \quad
\mathbf{g^{j,*}}(t,-\eta)=\mathbf{g^{j,*}}(t,\eta)\,.
\end{gather*}
By Lemma~\ref{lem:orth-relation},
$$\mathbf{g^\pm}(t,\eta),\mathbf{g^{\mp,*}}(t,\eta)\ra=-2\mu(\mp\eta)\,,
\quad
\la \mathbf{g^\pm}(t,\eta),\mathbf{g^{\pm,*}}(t,\eta)\ra=0\,,$$
\begin{align*}
&
\la \mathbf{g^1}(t,\eta),\mathbf{g^{1,*}}(t,\eta)\ra
=\la \mathbf{g^2}(t,\eta),\mathbf{g^{2,*}}(t,\eta)\ra
=-4\csch\k\Re\mu(\eta)=-4+O(\eta^2)\,,
\\ &
\la \mathbf{g^1}(t,\eta),\mathbf{g^{2,*}}(t,\eta)\ra
=4\csch\k\eta\Im\mu(\eta)=O(\eta^2)\,,
\\ &
\la \mathbf{g^2}(t,\eta),\mathbf{g^{1,*}}(t,\eta)\ra
=\la \mathbf{g^2}(t,\eta),\mathbf{g^{2,*}}(t,\eta)\ra
=4\frac{\Im\mu(\eta)}{\eta\sinh\k}  
=4\frac{\cosh\k}{\sinh^2\k}+O(\eta^2)\,.
\end{align*}
Let
\begin{gather*}
\mathcal{A}(t,\eta)=
  \begin{pmatrix}
    \la \mathbf{g^1}(t,\eta),\mathbf{g^{1,*}}(t,\eta)\ra
  &  \la \mathbf{g^2}(t,\eta),\mathbf{g^{1,*}}(t,\eta)\ra
  \\ 
  \la \mathbf{g^1}(t,\eta),\mathbf{g^{2,*}}(t,\eta)\ra
  &  \la \mathbf{g^2}(t,\eta),\mathbf{g^{2,*}}(t,\eta)\ra
\end{pmatrix}\,,
\\
\mathcal{P}(t,\eta)\mathbf{f}=(\mathbf{g^1}(t,\eta),\mathbf{g^2}(t,\eta))
\mathcal{A}(t,\eta)^{-1}
  \begin{pmatrix}
   \left \la \mathcal{F}_y\mathbf{f}(\eta),\mathbf{g^{1,*}}(t,\eta)\right\ra
 \\ \left \la \mathcal{F}_y\mathbf{f}(\eta),\mathbf{g^{2,*}}(t,\eta)\right\ra
\end{pmatrix}\,,
\\
  P_1(t,\eta_0)\mathbf{f}=
  \int_{-\eta_0}^{\eta_0}
  \mathcal{P}(t,\eta)\mathbf{f}e^{iy\eta}\,d\eta\,.
\end{gather*}
By Claim~\ref{cl:H1-bound}, we have
for $\eta_0\in[0,\eta_*(\a))$ and $t\in\R$,
\begin{gather*}
\sum_{j=1,2}\left(\|g^j(t,\eta)\|_{\ell^2_\a L^\infty(-\eta_0,\eta_0)}
  + \|g^{j,*}(t,\eta)\|_{\ell^2_{-\a}L^\infty(-\eta_0,\eta_0)}\right)  <\infty\,,
\end{gather*}
and $P_1(t,\eta_0)$ is a bounded operator on $\ell^2_\a H^1(\R)\times \ell^2_\a L^2(\R)$.
Moreover, if $0< \eta_0\le \eta_1<\eta_*(\a)$, there exists a
$K$ depending only on $\eta_0$, $\eta_1$ and $\a$ such that
\begin{equation*}
  \sup_{t\in\R} \left\|
\{P_1(t,\eta_1)-P_1(t,\eta_0)\}\mathbf{f}
  \right\|_{\ell^2_\a H^1(\R)\times \ell^2_\a L^2(\R)}
\le K\|\mathbf{f}\|_{\ell^2_\a L^2(\R)\times \ell^2_\a L^2(\R)}\,.
\end{equation*}

\par
We remark that $P_1(t,\eta)$ is commutative with a flow generated by \eqref{eq:6}.
Let $U(t,s)$ be an operator defined by
\begin{equation*}
  U(t,s) \begin{pmatrix}   f_1 \\ f_2 \end{pmatrix}
  =
  \begin{pmatrix}
    \bQd(t) \\ \pd_t\bQd(t)
  \end{pmatrix}\,,
\end{equation*}
where $\bQd(t)$ is a solution of \eqref{eq:6} satisfying $\bQd(s)=f_1$ and $\pd_t\bQd(s)=f_2$.
Then $U(t,s)P_1(s,\eta_0)=P_1(t,\eta_0)U(t,s)$.
\par
Now, we are in a position to prove Theorem~\ref{thm:main}.
\begin{proof}[Proof of Theorem~\ref{thm:main}]
First, we will show that if $|\eta|$ is sufficiently large,
a solution of \eqref{eq:linear1-F} decays at the same rate
as the corresponding solution of \eqref{eq:linear0-F}.
\par
Let $\bQ_\eta$ be a solution of \eqref{eq:linear0-F} in the class
$C^2(\R;\ell^2_\a)$. 
By Corollary~\ref{cor:exp-0}, there exists a positive constant $K_1$
depending only on $\a$ such that for every $\a>0$ and 
$t$, $s\in\R$ with $t\ge s$,
\begin{equation}
  \label{eq:decay-Qeta}
  \begin{split}
&  e^{-\a ct}\left( \|\la \eta\ra\bQ_\eta(t)\|_{\ell^2_\a}
    +\|\pd_t\bQ_\eta(t)\|_{\ell^2_\a}\right)
\\  \le & K_1e^{-b_1(t-s)} e^{-\a cs}\left(
\|\la\eta\ra\bQ_\eta(s)\|_{\ell^2_\a}+\|\pd_t\bQ_\eta(s)\|_{\ell^2_\a}\right)\,,
  \end{split}
\end{equation}
where $c=\sinh\k/\k$ and $b_1=c\a-2\sinh(\a/2)$.
\par
Let $\eta_1$, $\eta_2$, $\a_1$ and $\a_2$ be positive numbers such that
\begin{gather*}
\eta_0<\eta_1\,,\quad 0<\a_1<\a<\a_2<2\k\,,\quad
\eta_*(\a_1)<\eta_1<\eta_*(\a)<\eta_2<\eta_*(\a_2)\,.
\end{gather*}
Suppose that $\pm\eta\ge \eta_2$. Then, it follows from
Lemmas~\ref{lem:correspondence} and \ref{lem:CC'-high}  that for any
solution $\bQd_\eta\in C(\R;\ell^2_\a)$ of \eqref{eq:linear1-F},
there exists a unique solution
$\bQ_\eta\in C(\R;\ell^2_\a)$ of \eqref{eq:linear0-F} such that
for every $t\in\R$,
$(\mathbf{Q_1}(\eta),\mathbf{Q_2}(\eta))=(\bQ_\eta,\pd_t\bQ_\eta)$ and
$(\mathbf{Q'_1}(\eta),\mathbf{Q_2'}(\eta))=(\bQd_\eta,\pd_t\bQd_\eta)$ satisfy
\eqref{eq:Darboux2-F} and 
\begin{align}
\label{eq:QsimQd}
K^{-1}\left( \|\la \eta\ra\bQ_\eta(t)\|_{\ell^2_\a}
    +\|\pd_t\bQ_\eta(t)\|_{\ell^2_\a}\right)
 \le &
\|\la \eta\ra\bQd_\eta(t)\|_{\ell^2_\a} +\|\pd_t\bQd_\eta(t)\|_{\ell^2_\a}
  \\ \le &  \notag
 K\left( \|\la \eta\ra\bQ_\eta(t)\|_{\ell^2_\a}
         +\|\pd_t\bQ_\eta(t)\|_{\ell^2_\a}\right)\,,
\end{align}
where $K$ is a positive constant that depends only on $\a$ and $\eta_2$.
Combining the above with \eqref{eq:decay-Qeta}, we have for $\pm\eta\ge \eta_2$ and
$t$, $s\in\R$ with $t\ge s$,
\begin{equation}
\label{eq:14}
\begin{split}
 & e^{-\a ct}\left( \|\la \eta\ra\bQd_\eta(t)\|_{\ell^2_\a}
    +\|\pd_t\bQd_\eta(t)\|_{\ell^2_\a}\right)
  \\ \le & K_1K^2e^{-b_1(t-s)}
e^{-\a cs}\left( \|\la \eta\ra\bQd_\eta(s)\|_{\ell^2_\a}
    +\|\pd_t\bQd_\eta(s)\|_{\ell^2_\a}\right)\,.
\end{split}  
\end{equation}
Similarly, we have for $\pm\eta\ge \eta_1$ and $t$, $s\in\R$ with $t\ge s$,
\begin{equation}
  \label{eq:15}
  \begin{split}
&  e^{-\a_1 ct}\left( \|\la \eta\ra\bQd_\eta(t)\|_{\ell^2_{\a_1}}
    +\|\pd_t\bQd_\eta(t)\|_{\ell^2_{\a_1}}\right)
  \\ &
  \le  K_1'e^{-b_1'(t-s)}
e^{-\a_1 cs}\left( \|\la \eta\ra\bQd_\eta(s)\|_{\ell^2_{\a_1}}
  +\|\pd_t\bQd_\eta(s)\|_{\ell^2_{\a_1}}\right)\,, 
  \end{split}
\end{equation}
where $b_1'=c\a_1-2\sinh(\a_1/2)$ and $K_1'$ is a positive constant depending only on
$\a_1$ and $\eta_1$.
\par
If $|\eta|$ is sufficiently small, then solutions of \eqref{eq:linear1-F}
satisfying the secular term conditions decay like solutions of
\eqref{eq:linear0-F}.
Suppose that $\bQd_\eta$ is a solution of \eqref{eq:linear1-F}
in the class $C^2(\R;\ell^2_\a)$ satisfying
\eqref{eq:orth-Q1'Q2'} and \eqref{eq:orth-Q1'Q2'-new}.
If $\eta\in[-\eta_1,\eta_1]$, it follows from
Lemmas~\ref{lem:correspondence}, \ref{lem:sol-Q} and \ref{lem:C'-low} that
there exists a unique solution $\bQ_\eta$ of \eqref{eq:linear0-F}
satisfying \eqref{eq:QsimQd}, where $K_2$ is a constant that depends only on $\a$.
Thus for $\eta\in[-\eta_1,\eta_1]$, $\bQd_\eta$ satisfying
\eqref{eq:orth-Q1'Q2'} and \eqref{eq:orth-Q1'Q2'-new}
and $t$, $s\in\R$ with $t\ge s$,
\begin{equation}
\label{eq:16}
\begin{split}
& e^{-\a ct}\left( \|\la \eta\ra\bQd_\eta(t)\|_{\ell^2_\a}
    +\|\pd_t\bQd_\eta(t)\|_{\ell^2_\a}\right)
  \\ \le & K_2e^{-b_1(t-s)}
e^{-\a cs}\left( \|\la \eta\ra\bQd_\eta(s)\|_{\ell^2_\a}
    +\|\pd_t\bQd_\eta(s)\|_{\ell^2_\a}\right)\,,  
\end{split}
\end{equation}
where $K_2$ is a positive constant that depends only on $\a$ and $\eta_1$.

Likewise, if $\eta\in[-\eta_2,\eta_2]$ and $\bQd_\eta$ satisfies
\eqref{eq:orth-Q1'Q2'} and \eqref{eq:orth-Q1'Q2'-new},
then for $t$, $s\in\R$ with $t\ge s$,
\begin{equation}
  \label{eq:17}
  \begin{split}
&  e^{-\a_2 ct}\left( \|\la \eta\ra\bQd_\eta(t)\|_{\ell^2_{\a_2}}
    +\|\pd_t\bQd_\eta(t)\|_{\ell^2_{\a_2}}\right)
  \\ \le & K_2'e^{-b_1''(t-s)}
e^{-\a_2 cs}\left( \|\la \eta\ra\bQd_\eta(s)\|_{\ell^2_{\a_2}}
    +\|\pd_t\bQd_\eta(s)\|_{\ell^2_{\a_2}}\right)\,,
  \end{split}
\end{equation}
where $b_1''=c\a_2-2\sinh(\a_2/2)$ and $K_2'$ is a positive constant
depending only on $\a_2$ and $\eta_2$.
\par
For $\mathbf{f}\in \ell^2_\a H^1\times \ell^2_\a L^2$ and  $a_1$, $a_2$ with
$0\le a_1\le a_2\le \infty$, let 
\begin{equation*}
P_0(a_1,a_2)\mathbf{f}(n,y)=
\frac{1}{2\pi}  \int_{a_1\le |\eta|\le a_2}\int_\R \mathbf{f}(n,y_1)e^{i\eta(y-y_1)}\,dy_1d\eta\,.
\end{equation*}
We have $U(t,s)P_0(a_1,a_2)=P_0(a_1,a_2)U(t,s)$ since $V^\k$ is independent of $y$.
By \eqref{eq:10}, \eqref{eq:14} and \eqref{eq:15}, we have for $t$, $s\in\R$ with $t\ge s$,
\begin{gather}
 \label{eq:freqh}
  e^{-\a ct}\|P_0(\eta_2,\infty)U(t,s)\mathbf{f}\|_{\ell^2_\a H^1\times \ell^2_\a L^2}
\le K_3e^{-b_1(t-s)}
e^{-\a cs}\|\mathbf{f}\|_{\ell^2_\a H^1\times \ell^2_\a L^2}\,,
\\
  \label{eq:freq1-2}
 e^{-\a_1ct}
 \|P_0(\eta_1,\eta_2)U(t,s)\mathbf{f}\|_{\ell^2_{\a_1} H^1\times \ell^2_{\a_1} L^2}
 \le K_3'
e^{-b_1'(t-s)}e^{-\a_1 cs}\|\mathbf{f}\|_{\ell^2_{\a_1} H^1\times \ell^2_{\a_1} L^2}\,,
\end{gather}
where $K_3$ is a positive constant that depends only on $\a$ and $\eta_2$ and
$K_3'$ is a positive constant that depends only on $\a_1$ and $\eta_1$.
\par
By \eqref{eq:16}, we have for $t\ge s$,
\begin{equation}
  \label{eq:freql}
  \begin{split}
&  e^{-\a ct}\left\|\{P_0(0,\eta_0)-P_1(t,\eta_0)\}U(t,s)\mathbf{f}
\right\|_{\ell^2_\a H^1\times \ell^2_\a L^2}
\\ \le & K_4e^{-b_1(t-s)}
e^{-\a cs}\|\mathbf{f}\|_{\ell^2_\a H^1\times \ell^2_\a L^2}\,,
  \end{split}
\end{equation}
\begin{align*}
&  e^{-\a ct}\left\|\{P_0(\eta_0,\eta_1)-P_1(t,\eta_1)+P_1(t,\eta_0)\}U(t,s)\mathbf{f}
\right\|_{\ell^2_\a H^1\times \ell^2_\a L^2}
  \\ \le &
K_4e^{-b_1(t-s)}e^{-\a cs}\|\mathbf{f}\|_{\ell^2_\a H^1\times \ell^2_\a L^2}\,,
\end{align*}
where $K_4$ is a positive constant that depends only on $\a$ and $\eta_1$.
\par
On the other hand, it follows from Lemma~\ref{cl:delta-eta} that for $t\ge s$,
\begin{equation}
  \label{eq:19}
\begin{split}
&  e^{-\a ct}\left\|    \{P_1(t,\eta_1)-P_1(t,\eta_0)\}U(t,s)\mathbf{f}
\right\|_{\ell^2_\a H^1\times \ell^2_\a L^2}
\\  \le & K_5e^{-\delta_R(\eta_0)(t-s)}
e^{-\a cs}\|\mathbf{f}\|_{\ell^2_\a H^1\times \ell^2_\a L^2}\,,
\end{split}  
\end{equation}
where $K_5$ is a constant that depends only on $\eta_0$, $\eta_1$ and $\a$.
Combining the above, we have for $t\ge s$,
\begin{gather*}
  e^{-\a ct}\left\|P_0(\eta_0,\eta_1)U(t,s)\mathbf{f}
    \right\|_{\ell^2_\a H^1\times \ell^2_\a L^2} \le K_6e^{-b_2(t-s)}
e^{-\a cs}\|\mathbf{f}\|_{\ell^2_\a H^1\times \ell^2_\a L^2}\,,
\end{gather*}
where $K_6$ and $b_2$ are positive constants that depend only on $\eta_0$, $\eta_1$ and $\a$.
Similarly, it follows from \eqref{eq:17} and Lemma~\ref{cl:delta-eta}
that for $t\ge s$,
\begin{equation}
  \label{eq:freq1-2'}
  e^{-\a_2 ct}\left\|P_0(\eta_1,\eta_2)U(t,s)\mathbf{f}
  \right\|_{\ell^2_{\a_2} H^1\times \ell^2_{\a_2} L^2} \le K_6'e^{-b_2'(t-s)}
 e^{-\a_2 cs}\|\mathbf{f}\|_{\ell^2_{\a_2} H^1\times \ell^2_{\a_2} L^2}\,,
\end{equation}
where $K_6'$ and $b_2'$ are positive constants that depend only on $\eta_1$, $\eta_2$ and $\a_2$.
\par
Applying the complex interpolation theorem to \eqref{eq:freq1-2} and \eqref{eq:freq1-2'},
we have for $t\ge s$,
\begin{equation}
 \label{eq:freqt}
  e^{-\a ct}\left\|P_0(\eta_1,\eta_2)U(t,s)\mathbf{f}
  \right\|_{\ell^2_\a H^1\times \ell^2_\a L^2} \le K_7e^{-b_3(t-s)}
 e^{-\a cs}\|\mathbf{f}\|_{\ell^2_\a H^1\times \ell^2_\a L^2}\,,
\end{equation}
where $K_7$ and $b_3$ are positive constants depending only on
$\a$, $\a_1$, $\a_2$, $\eta_1$ and $\eta_2$.
\par
Combining \eqref{eq:freqh}, \eqref{eq:freql}, \eqref{eq:19}
and \eqref{eq:freqt}, we have
Theorem~\ref{thm:main}. This completes the proof of Theorem~\ref{thm:main}.
\end{proof}

Finally, we will prove Theorem~\ref{thm:profile}.
\begin{proof}[Proof of Theorem~\ref{thm:profile}]
\par
It follows from Theorem~\ref{thm:main} that for $t\ge s$,
\begin{align*}
& e^{-\a ct}  \left\|\left(I-P(t,\eta_0)\right)U(t,s)\mathbf{f}
\right\|_{\ell^2_\a H^1(\R)\times \ell^2_\a L^2(\R)}
  \\  \le & K
e^{-b(t-s)}e^{-\a cs} \|\mathbf{f}\|_{\ell^2_\a H^1(\R)\times \ell^2_\a L^2(\R)}\,,
\end{align*}
where $c=\sinh\k/\k$ and $K$ and $b$ are positive constants.
\par
By Lemma~\ref{cl:delta-eta},
\begin{equation*}
\delta_R(\eta)=\lambda_2\eta^2+O(\eta^4)\,,\quad
\delta_I(\eta)=-\lambda_1\eta+O(\eta^3)\,.
\end{equation*}

\begin{gather}
  \label{eq:exp dR}
  \|e^{-t\delta_R(\eta)}\|_{L^2(-\eta_0,\eta_0)}=O(t^{-1/4})\,,
\\ \label{eq:exp dR-sin}
\left\|e^{-t\delta_R(\eta)}\frac{\sin t\delta_I(\eta)}{\eta}
+e^{-t\lambda_2\eta^2}\frac{\sin t\lambda_1\eta}{\eta}
\right \|_{L^2(-\eta_0,\eta_0)}=O(t^{-1/4})\,.
\end{gather}
By Claim~\ref{cl:H1-bound}, \eqref{eq:tg1}--\eqref{eq:tg2*},
\eqref{eq:exp dR} and \eqref{eq:exp dR-sin},
\begin{gather*}
\|g^1(t,\eta)\|_{\ell^2_\a L^2(-\eta_0,\eta_0)}=O(t^{-1/4})\,,
\\
 \left\|\tg^2(t,\eta)+e^{-t\lambda_2\eta^2-\k z_n(t)}\sech \k z_n(t)\frac{\sin t\lambda_1\eta}{\eta}
\right \|_{\ell^2_\a L^2(-\eta_0,\eta_0)}=O(t^{-1/4})\,.
\end{gather*}
In the last line, we use
\begin{equation}
  \label{eq:21}
\gamma_R(\eta)=\k+O(\eta^2)\,,\quad \gamma_I(\eta)=\eta\csch\k+O(\eta^3)\,.
\end{equation}
Let
$$Q^\k_n=\log\frac{\cosh\k z_n(t)}{\cosh\k z_{n-1}(t))}\,.$$
Then $R^\k_n=Q^\k_{n+1}-Q^\k_n$ and
\begin{align*}
\pd_tQ^\k_n=& -\sinh\k (1-e^{-\pd})\tanh\k z_n(t)
  \\=&
\sinh\k (1-e^{-\pd})\left(e^{-\k z_n(t)}\sech \k z_n(t)\right)
\\=& -\sinh^2\k\sech\k z_n(t)\sech\k z_{n-1}(t)\,.  
\end{align*}
Thus we have
\begin{align*}
\left\|g^2(t,\eta)+e^{-t\lambda_2\eta^2}\frac{\sin t\lambda_1\eta}{\eta}
  \csch\k\pd_tQ^\k\right \|_{\ell^2_\a L^2(-\eta_0,\eta_0)}=O(t^{-1/4})\,.
\end{align*}
\par

Suppose that $\bQd(t)$ is a solution of \eqref{eq:6}. Then
by Lemma~\ref{lem:qJ-1q*},
\begin{align*}
  \left\la
\begin{pmatrix}
  \mF_y\bQd(t)(\eta) \\   \mF_y\pd_t\bQd(t)(\eta)
\end{pmatrix},
\mathbf{g^{j,*}}(t,\eta)\right\ra
  =&  \left\la
\begin{pmatrix}
     \mF_y\bQd(0)(\eta) \\  \mF_y\pd_t\bQd(0)(\eta)
\end{pmatrix},
\mathbf{g^{j,*}}(0,\eta)\right\ra
  \\=& \hat{f}(\eta)++O(\eta)\,,
\end{align*}
where
\begin{equation*}
  f_j(y)=\left\la
\begin{pmatrix}
     \bQd(0) \\ \pd_t\bQd(0)
\end{pmatrix},
\mathbf{g^{j,*}}(0,0)\right\ra\,.
\end{equation*}
By \eqref{eq:tg1*}, \eqref{eq:tg2*} and \eqref{eq:21},
\begin{equation*}
g^{1,*}(0,0)=-\csch\k(1+\pd_kQ^\k)\in\ell^2_{-\a}\,,\quad
g^{2,*}(0,0)=-\csch\k\pd_tQ^\k\in\ell^2_{-\a}\,,
\end{equation*}
and $f_j\in L^1(\R_y)$ for $j=1$ and $2$.
\end{proof}

\appendix

\section{Miscellaneous results}
\label{sec:appendix1}
Let $\a\in\R$ and $\ell^2_\a$ be a Hilbert of complex sesuences
with a norm $\|x\|_{\ell^2_\a}=(\sum_{n\in\Z}e^{2\a n}|x_n|^2)^{1/2}$
for $x=\{x_n\}_{n\in\Z}$.
The operators $e^{\pm\pd}-1$ have bounded inverse on $\ell^2_\a$
if $\a\ne0$. Indeed, we have the following.
\begin{lemma}
  \label{lem:pd-1}
Let $\a$ be a positive constant. Then
\begin{gather}
  \label{eq:pd-1a}
    \|e^{\pm\pd} f\|_{\ell^2_\a}=e^{\mp\a}\|f\|_{\ell^2_\a}\,,\\
  \label{eq:pd-1b}
    \|(e^\pd-1)^{-1}f\|_{\ell^2_\a}\le \frac{1}{1-e^{-\a}}\|f\|_{\ell^2_\a}\,,
    \quad
    \|(1-e^{-\pd})^{-1}f\|_{\ell^2_\a}\le \frac{1}{e^\a-1}\|f\|_{\ell^2_\a}\,.
  \end{gather}
\end{lemma}
\begin{proof}
 Eq.~\eqref{eq:pd-1a} follows immediately from the definitions.
 Since
 \begin{equation*}
(e^\pd-1)^{-1}=-\sum_{k\ge0}e^{k\pd}\,,\quad
(1-e^{-\pd})^{-1}=-\sum_{k\ge1}e^{k\pd}\quad\text{on $\ell^2_\a$,}
 \end{equation*}
 we have \eqref{eq:pd-1b} from \eqref{eq:pd-1a}.
\end{proof}

To estimate the $\ell^2_\a$-norm of Jost functions, we use the
 following.
 \begin{claim}
   \label{cl:H1-bound}
   For every $f\in H^1(\R)$,
   $\sum_{n\in\Z}|f(n)|^2\le 2\|f\|_{H^1(\R)}^2$.
 \end{claim}
 \begin{proof}
We may assume that $f$ is real-valued and
$f\in C^\infty_0(\R)$.
  For $x\in[n,n+1]$,
   \begin{align*}
     f(n)^2=& f(x)^2+2\int_x^n f(y)f'(y)\,dy
              \\ \le & f(x)^2+\int_n^x\{f'(y)^2+f(y)^2\}\,dy\,.
   \end{align*}
   Integrating the above over $[n,n+1]$ and adding the resulting equation
for each $n\in\Z$, we have
   \begin{align*}
 \sum_{n\in\Z} f(n)^2\le & \sum_{n\in\Z} \int_n^{n+1}\{f'(x)^2+2f(x)^2\}\,dx
\le 2\|f\|_{H^1(\R)}^2\,.
   \end{align*}
   Thus we complete the proof.
 \end{proof}

Let $\beta_+$ and $\beta_-$ be complex constants
satisfying $|\beta_+|<1<|\beta_-|$.
We will show that for any $f=\{f_n\}_{n\in\Z}\in \ell^p$, there exists
$a=\{a_n\}_{n\in\Z}\in \ell^p$ such that
\begin{equation*}
  a_{n+1}-(\beta_++\beta_-)a_n+\beta_+\beta_-a_{n-1}=f_n\quad\text{for
    $n\in\Z$.}
\end{equation*}
For a complex sequence $f=\{f_n\}_{n\in\Z}$, let
\begin{equation*}
 k*f(n):=\sum_{m\in\Z} k_mf_{n-m}\,.
\end{equation*}
Then, we have the following.
\begin{lemma}
  \label{lem:fundsol-seq}
Let $1\le p\le\infty$ and $k=\{k_n\}_{n\in\Z}$,
$k_n=\beta_+^n/(\beta_+-\beta_-)$ for $n\ge0$ and
$k_n=\beta_-^n/(\beta_+-\beta_-)$ for $n\le-1$. 
If $|\beta_+|<1<|\beta_-|$ and $f\in\ell^p$,
\begin{equation*}
  \|k*f\|_{\ell^p}\le \frac{1}{|\beta_+-\beta_-|}
  \left(\frac{1}{1-|\beta_+|}
    +\frac{1}{|\beta_-|-1}\right)\|f\|_{\ell^p}\,,
\end{equation*}
and $\mathbf{a}=k*f$ is a solution of
\begin{equation}
  \label{eq:sol-f}
  (e^\pd-\beta_+-\beta_-+\beta_+\beta_-e^{-\pd})\mathbf{a}=f\,.
\end{equation}
  \end{lemma}
  \begin{proof}
By the definition,
    \begin{equation}
 \label{eq:fund-sol}
k_{n+1}-(\beta_++\beta_-)k_n+\beta_+\beta_-k_{n-1}=\delta_{0n}\,,
\end{equation}
and  \eqref{eq:sol-f} follows from \eqref{eq:fund-sol}.
\par
Moreover,
\begin{equation*}
  \|k\|_{\ell^1}=\frac{1}{|\beta_+-\beta_-|}
  \left(\frac{1}{1-|\beta_+|}+\frac{1}{|\beta_-|-1}\right)\,,
\end{equation*}
and $\|\mathbf{a}\|_{\ell^p}\le \|k\|_{\ell^1}\|f\|_{\ell^p}$
by Young's inequality.
Thus we complete the proof.
\end{proof}

\begin{lemma}
  \label{lem:qJ-1q*}
Let $\a>0$, $\eta\in\R$ and $V'_n(t)\in C(\R;\ell^\infty)$.
Let $q(t)\in C^2(\R,\ell^2_\a)$ and $q^*(t)\in C^2(\R;\ell^2_{-\a})$ be solutions of
 \begin{equation}
    \label{eq:1}
(\pd_t^2+\eta^2)Q=(1-e^{-\pd})(1+V')(e^\pd-1)Q\,.
  \end{equation}
Then $\la q(t),\pd_tq^*(t)\ra-\la \pd_tq(t),q^*(t)\ra$
does not depend on $t$.
\end{lemma}
\begin{proof}
  Let $\mathbf{q}={}^t\!(q,\pd_tq)$ and $\mathbf{q^*}={}^t\!(q^*,\pd_tq^*)$.
  Let $H_1=\eta^2-(1-e^{-\pd})(1+V')(e^\pd-1)$ and
  \begin{gather*}
    J= \begin{pmatrix}   O & I \\ -I & O  \end{pmatrix}\,,
    \quad
    H=
    \begin{pmatrix} H_1 & O \\ O & I  \end{pmatrix}\,.
  \end{gather*}
Then $d\mathbf{q}/dt=JH\mathbf{q}$ and $d\mathbf{q^*}/dt=JH\mathbf{q^*}$
Since $J^*=-J$ and $H^*=H$,
\begin{align*}
\frac{d}{dt}\la \mathbf{q},J^{-1}\mathbf{q^*}\ra
 =& \la JH\mathbf{q}, J^{-1}\mathbf{q^*}\ra + \la \mathbf{q}, H\mathbf{q^*}\ra
\\ = & 0\,,
\end{align*}
and
\begin{equation*}
  \la \mathbf{q},J^{-1}\mathbf{q^*}\ra
  =\sum_{n\in\Z}(\pd_tq_nq^*_n-q_n\pd_tq^*_n)
\end{equation*}
does not depend on $t$.
\end{proof}

By the definitions \eqref{def:Phi0}--\eqref{eq:dualJost}, we have the following.
\begin{claim}
  \label{cl:Phi-updown}
\begin{equation*}
  \Phi(a)=ue^{-\pd}\Phi(a)=e^\pd\left(v\Phi(a)\right)
= \left(a-\frac{1}{a}\right)
    \frac{e^{(a+1/a)(x-s)}}{\tau'}\,. 
\end{equation*}
 For any $\beta\in\C$,
 \begin{gather}
   \notag
   \Phi^0(\beta)-ve^{-\pd}\Phi^0(\beta)=e^{-\pd}\Phi(\beta)\,,
   \\ \label{eq:Phi-updown-1}
(e^\pd-1)\left(\frac{\Phi^0(\beta)}{\tau'}\right)
=\frac{\Phi(\beta)}{e^\pd\tau'}\,.  
\end{gather}
For any $\beta\in\C\setminus\{0,a,1/a\}$,
\begin{gather}
\Phi^*(\beta)=e^{-\pd}\Phi^{0,*}(\beta)-(e^\pd u)\Phi^{0,*}(\beta)\,,
\\ \label{eq:Phi-updown-2}
e^{-\pd}\Phi^*(\beta)-(e^\pd v)\Phi^*(\beta)=\Phi^{0,*}(\beta)\,,
\\
\Phi(a)\Phi^*(\beta)=-\frac{1}{(\beta-a)(\beta-\frac1a)}
(e^\pd-1)\left(\Phi(a)e^{-\pd}\Phi^{0,*}(\beta)\right)\,.
\end{gather}
Especially,
\begin{gather*}
    \Phi(a)\Phi^*(\beta_\pm(\eta))=
    \frac{1}{2i\eta}
    (e^\pd-1)\left(\Phi(a)\Phi^{0,*}(\beta_\pm(\eta))\right)\,.
  \end{gather*}
\end{claim}
\bigskip

\section{Orthogonality relation of $g^\pm(\eta)$ and $g^{\pm,*}(\eta)$}
\label{sec:appendix2}
To prove Claim~\ref{cl:pd-t,AB} and Lemma~\ref{lem:orth-relation},
we need the following.
\begin{claim}
  \label{cl:t,y-JdJ}
Let $L_1$ and $L_2$ be as \eqref{eq:Laxpair}.
If  $L_1\Phi=L_2\Phi=0$ and $L_1^*\Phi^*=L_2^*\Phi^*=0$,  then
\begin{gather*}
  2\pd_t(\Phi\Phi^*)=
  (e^\pd-1)\{\Phi(e^{-\pd}\Phi^*)+(1+V)(e^{-\pd}\Phi)\Phi^*\}\,,
  \\
    2\pd_y(\Phi\Phi^*)=
  (e^\pd-1)\{\Phi(e^{-\pd}\Phi^*)-(1+V)(e^{-\pd}\Phi)\Phi^*\}\,.
\end{gather*}
\end{claim}
Claim~\ref{cl:t,y-JdJ} follows from
\eqref{eq:v-change}, \eqref{eq:Jost-sol} and \eqref{eq:dualJost-sol}.

\par
Now we will prove Claim~\ref{cl:pd-t,AB}.
\begin{proof}[Proof of Claim~\ref{cl:pd-t,AB}]
 Suppose that $Q=e^{-\pd}\left(\Phi(\beta_1)\Phi^*(\beta_2)\right)$
 for $\beta_1$, $\beta_2\in\C$  and that $\pd_yQ=i\eta Q$.
Note that $\pd_yg^\pm(\eta)=i\eta g^\pm(\eta)$ and 
$\pd_yg^{\pm,*}(\eta)=i\eta\tg^{\pm,*}(\eta)$ by Lemma~\ref{lem:tgform}.
\par
 
By Claims~\ref{cl:t,y-JdJ} and \ref{cl:Phi-updown},
 \begin{align*}
e^\pd B'(i\eta)Q+\pd_tQ=& (\pd_t+\pd_y)Q-(e^\pd-1)vQ
   \\=& (1-e^{-\pd})\{\Phi(\beta_1)e^{-\pd}\Phi^*(\beta_2)
        -(e^{\pd}v)\Phi(\beta_1)\Phi^*(\beta_2)\}
   \\=&
        (1-e^{-\pd})\{\Phi(\beta_1)\Phi^{0,*}(\beta_2)\}\,.
 \end{align*}
 By the definitions, $A'(-i\eta)-B'(i\eta)+C'(\eta)=(e^\pd-1)B'(i\eta)$.
 Hence it follows that
 \begin{align}
   \label{eq:5}
& \{A'(-i\eta)-B'(i\eta)+C'(\eta)\}Q+(1-e^{-\pd})\pd_tQ
\\ = & \notag
(1-e^{-\pd})\{e^\pd B'(i\eta)Q+\pd_tQ\}
\\=&  \notag
   (1-e^{-\pd})^2\Phi(\beta_1)\{\Phi^{0,*}(\beta_2)\}\,.   
 \end{align}
Let  $\beta_1=\beta_+(-\eta)$ and 
calculate the residue of \eqref{eq:5} at $\beta_2=a$.
Since $\operatorname{Res}_{\beta=a}\Phi^*(\beta)=1/(e^\pd\tau')$,
$g^+=e^{-\pd}\Phi(\beta_1)/\tau'$ and 
$C'(\eta)g^+(\eta)=0$ by  Lemma~\ref{lem:kerC-C'},
we have \eqref{eq:clpd-t,AB-1}.
\par
Substituting $\beta_1=a$ and $\beta_2=\beta_-(\eta)$ into \eqref{eq:5},
we have
\begin{align}
  \label{eq:clpd-t,AB-3'}
&  \left\{A'(-i\eta)-B'(i\eta)+C'(\eta)\right\}g^-(\eta)
  +(1-e^{-\pd})\pd_tg^-(\eta) \\
  \\ =& \notag
(1-e^{-\pd})^2\left\{\Phi(a)\Phi^{0,*}(\beta_-(\eta))\right\}\,.  
\end{align}
By Lemma~\ref{lem:Darboux}, Claim~\ref{cl:Phi-updown},
\eqref{def:gpm} and \eqref{def:tgpm*},
\begin{align*}
&  -C'(\eta)g^-(\eta)+(1-e^{-\pd})^2\left(\Phi(a)\Phi^{0,*}(\beta_-(\eta))\right)
\\ =&
(A'-e^{\pd}B')e^{-\pd}\left\{\Phi(a)\Phi(\beta_-(\eta))\right\}
+(1-e^{-\pd})^2\left(\Phi(a)\Phi^{0,*}(\beta_-(\eta))\right)
\\ =& -2i\eta(1-e^{-2\pd})\tg^-(\eta)\,.
\end{align*}
Combining the above with \eqref{eq:clpd-t,AB-3'}, we have
\eqref{eq:clpd-t,AB-3}.
\par

Since $C(\eta)=A(-i\eta)-B(i\eta)$, it follows from
Claim~\ref{cl:A,B,dual} that
 \begin{align*}
   \{A'(-i\eta)-B'(i\eta)-C(\eta)\}^*(e^\pd-1)^{-1}=
   i\eta+(1-e^{-\pd})u\,.
 \end{align*}
Since $(i\eta-\pd_t)Q=(\pd_y-\pd_t)Q=-\pd_sQ$,
 it follows from Lemma~\ref{lem:Darboux} that
 \begin{align}
   \label{eq:18}
& \{A'(-i\eta)-B'(i\eta)-C(\eta)\}^*(e^\pd-1)^{-1}Q-\pd_tQ
=  -A'Q
\\ =& \notag
 (1-e^{-\pd})\left\{u\left(e^{-\pd}\Phi(\beta_1)\right)
       \Phi^{0,*}(\beta_2)\right\}\,.
\end{align}
Letting $\beta_1=\beta_-(-\eta)$ and
calculating the residue of the equation above at $\beta_2=a$, we have
\eqref{eq:clpd-t,AB-2} from \eqref{def:g*pm}, \eqref{eq:g-tg} and
the fact that $C(\eta)^*e^{-\pd}\tg^{+,*}(\eta)=0$.
\par
Substituting $\beta_1=a$ and $\beta_2=\beta_+(\eta)$ into \eqref{eq:18}
and using \eqref{def:g*pm}, \eqref{def:tgpm*} and \eqref{eq:g-tg}, we have
\begin{align}
 \label{eq:clpd-t,AB-4'}
&  \left\{A'(-i\eta)-B'(i\eta)-C(\eta)\right\}^*e^{-\pd}\tg^{-,*}(\eta)
  -\pd_tg^{-,*}(\eta)
  \\ = & \notag
2i\eta(1-e^{-\pd})\tg^{-,*}(\eta)\,.
\end{align}
By the definitions, $C(\eta)^*(e^\pd-1)^{-1}=e^\pd B'-A'$ and
it follows from Lemma~\ref{lem:Darboux} that
\begin{align*}
C(\eta)^*e^{-\pd}\tg^{-,*}(\eta)=&
(e^\pd-1)^{-1}(e^\pd B'-A')e^{-\pd}\left(\Phi(a)\Phi^*(\beta_+(\eta))\right)
  \\=& 2e^{-\pd}\left(\Phi(a)\Phi^*(\beta_+(\eta))\right)
= 4i\eta e^{-\pd}\tg^{-,*}(\eta) \,.
\end{align*}
Substituting the above into \eqref{eq:clpd-t,AB-4'},
we have \eqref{eq:clpd-t,AB-4}.  
Thus we complete the proof.
\end{proof}
Finally, we will prove Lemma~\ref{lem:orth-relation}.
\begin{proof}[Proof of Lemma~\ref{lem:orth-relation}]
Since $g^{+,*}(\eta)=(1-e^{-\pd})\tg^{+,*}(\eta)$, it follows from  
\eqref{eq:clpd-t,AB-1} and \eqref{eq:clpd-t,AB-2} that
  \begin{align*}
& \la g^+(\eta),\pd_tg^{+,*}(\eta)\ra-\la \pd_tg^+(\eta),g^{+,*}(\eta)\ra
  \\=&
\la g^+(\eta),\{A'(-i\eta)-B'(i\eta)\}^*e^{-\pd}\tg^{+,*}(\eta)\ra
\\ & +\la (1-e^{-\pd})^{-1}\{A'(-i\eta)-B'(i\eta)\}g^+(\eta),g^{+,*}(\eta)\ra  
\\=& 0\,.
\end{align*}
\par
By \eqref{eq:g-tg}, \eqref{eq:clpd-t,AB-3} and \eqref{eq:clpd-t,AB-4},
\begin{align*}
& \la g^-(\eta),\pd_tg^{-,*}(\eta)\ra-\la \pd_tg^-(\eta),g^{-,*}(\eta)\ra  
\\=&
2i\eta\left\{
 \left\la g^-(\eta), (1+e^{-\pd})\tg^{-,*}(\eta)\right\ra 
 +\left\la (1+e^{-\pd})\tg^-(\eta), g^{-,*}(\eta)\right\ra
       \right\}
  \\=& -2i\eta\left\{
 \left\la \tg^-(\eta), (e^\pd-e^{-\pd})\tg^{-,*}(\eta)\right\ra 
 +\left\la (e^\pd-e^{-\pd})\tg^-(\eta), \tg^{-,*}(\eta)\right\ra
       \right\}
\\= & 0\,.
\end{align*}
\par
By \eqref{eq:clpd-t,AB-2}, \eqref{eq:clpd-t,AB-3} and \eqref{eq:g-tg},
\begin{align}
 \label{eq:20} 
\la g^-(\eta),\pd_tg^{+,*}(\eta)\ra
    -\la \pd_tg^-(\eta),g^{+,*}(\eta)\ra
=&  -2i\eta\la (e^\pd-e^{-\pd})\tg^-(\eta),\tg^{+,*}(\eta)\ra\,.
\end{align}
Since $\Phi^0(a)\Phi^{0,*}(\beta_\pm(\eta))=a^n\beta_\mp(\eta)^n
e^{iy\eta-t\{\sinh\k\pm\mu(\eta)\}}$,
\begin{equation*}
\pd_y\{\Phi^0(a)\Phi^{0,*}(\beta_\pm(\eta))\}=i\eta
\Phi^0(a)\Phi^{0,*}(\beta_\pm(\eta))\,.  
\end{equation*}
Hence it follows from Lemma~\ref{lem:Darboux} and Claim~\ref{cl:Phi-updown}
that
\begin{align*}
e^\pd C(\eta)e^{-\pd}\left\{\Phi^0(a)\Phi^{0,*}(\beta_\pm(\eta))\right\}
=& e^\pd(A-B)e^{-\pd}\left\{\Phi^0(a)\Phi^{0,*}(\beta_\pm(\eta))\right\}
\\=& -(e^\pd-e^{-\pd})\left\{\Phi(a)\Phi^{0,*}(\beta_\pm(\eta))\right\}\,.
\end{align*}
Combining the above with \eqref{def:tgpm} and \eqref{def:tgpm*}, we have
\begin{gather}
\label{eq:21a}
e^\pd C(\eta)e^{-\pd}\left\{\Phi^0(a)\Phi^{0,*}(\beta_-(\eta))\right\}  
=-2i\eta(e^\pd-e^{-\pd})\tg^-(\eta)\,,
\\ \label{eq:21b}
e^\pd C(\eta)e^{-\pd}\left\{\Phi^0(a)\Phi^{0,*}(\beta_+(\eta))\right\}  
=-2i\eta(e^\pd-e^{-\pd})\tg^{-,*}(\eta)\,.
\end{gather}
\par
By \eqref{eq:20} and \eqref{eq:21a},
\begin{align*}
& \la g^-(\eta),\pd_tg^{+,*}(\eta)\ra
    -\la \pd_tg^-(\eta),g^{+,*}(\eta)\ra
\\ =&
\left\la e^\pd C(\eta)e^{-\pd}\left\{\Phi^0(a)\Phi^{0,*}(\beta_-(\eta))\right\},
\tg^{+,*}(\eta)\right\ra\,.
\end{align*}
By Lemma~\ref{lem:kerC-C'},
\begin{equation*}
\left\{e^\pd C(\eta)e^{-\pd}\right\}^*\tg^{+,*}(\eta)
=e^\pd C(\eta)^*e^{-\pd}\tg^{+,*}=0\,.
\end{equation*}
Since $e^\pd C(\eta)e^{-\pd}=-2i\eta+u+(e^\pd v)-e^\pd u-ve^{-\pd}$,
\begin{equation}
    \label{eq:9}
  (e^\pd Ce^{-\pd}f)\overline{g}-f\overline{e^\pd C^*e^{-\pd}g}
  =(e^\pd-1)\{v(e^{-\pd}f)\bar{g}-ufe^{-\pd}\bar{g}\}\,.
\end{equation}
By Claim~\ref{cl:Phi-updown} and \eqref{def:taun'},
\begin{equation*}
\Phi^0_n(a)\Phi^{0,*}_n(\beta_-(\eta))
\overline{\Phi^0_n(\beta_-(-\eta))}/\tau'_n
=O(e^{-2\k|n|})\quad\text{as $n\to-\infty$.}
\end{equation*}
Hence it follows from \eqref{def:tgpm*} and \eqref{eq:9}  that
\begin{align*}
  & \left\la e^\pd C(\eta)e^{-\pd}
    \left\{\Phi^0(a)\Phi^{0,*}(\beta_-(\eta))\right\},
\tg^{+,*}(\eta)\right\ra
  \\ =& \lim_{n\to\infty}
\biggl\{v_{n+1}\Phi^0_n(a)\Phi^{0,*}_n(\beta_-(\eta))
\Phi^0_{n+1}(\beta_-(\eta))(\tau'_{n+1})^{-1}
  \\ & \qquad\qquad 
  -u_{n+1}\Phi^0_{n+1}(a)\Phi^{0,*}_{n+1}(\beta_-(\eta))
\Phi^0_n(\beta_-(\eta))(\tau'_n)^{-1}\biggr\}\,.
\end{align*}
By the definitions,
\begin{gather*}
\Phi^0_j(\beta)\Phi^{0,*}_k(\beta)=\beta^{j-k}
 \quad\text{for every $j$, $k\in\Z$,}  
\\
\lim_{n\to\infty}\frac{v_{n+1}\Phi^0_n(a)}{\tau'_{n+1}}
=\lim_{n\to\infty}\frac{\Phi^0_n(a)}{\tau'_n}=1\,,
  \\
\lim_{n\to\infty}\frac{u_{n+1}\Phi^0_{n+1}(a)}{\tau'_n}
=\lim_{n\to\infty}\frac{\Phi^0_{n+1}(a)}{\tau'_{n+1}}=1\,.
\end{gather*}
Combining the above, we have,
\begin{align*}
\la g^-(\eta),\pd_tg^{+,*}(\eta)\ra-\la \pd_tg^-(\eta),g^{+,*}(\eta)\ra  
=&\beta_-(\eta)-\beta_-(\eta)^{-1}
\\=& -2\mu(\eta)\,.
\end{align*}
\par
By \eqref{eq:clpd-t,AB-1}, \eqref{eq:clpd-t,AB-4}, \eqref{eq:21b}
and \eqref{eq:9},
\begin{align*}
& \la g^+(\eta),\pd_tg^{-,*}(\eta)\ra-\la \pd_tg^+(\eta),g^{-,*}(\eta)\ra  
  \\=& -\left\la \tg^+(\eta)), e^\pd C(\eta)e^{-\pd}
\left\{\Phi^0(a)\Phi^{0,*}(\beta_+(\eta))\right\}\right\ra       
  \\=& \lim_{n\to\infty}       
\Bigl\{-v_{n+1}\Phi^0_{n+1}(\beta_+(-\eta))(\tau'_{n+1})^{-1}
       \Phi^0_n(a)\Phi^{0,*}_n(-\beta_+(\eta))
  \\ &
       +u_{n+1}\Phi^0_n(\beta_+(-\eta))(\tau'_n)^{-1}
       \Phi^0_{n+1}(a)\Phi^{0,*}_{n+1}(\beta_+(-\eta))\Bigr\}
  \\=& -\beta_+(-\eta)+\beta_+(-\eta)^{-1}
 =-2\mu(-\eta)\,.
\end{align*}
Thus we complete the proof.
\end{proof}

\bigskip

\section*{Acknowledgments}
This work was supported by JSPS KAKENHI Grant Numbers JP21K03328 and 24H00185.

\end{document}